\documentclass[english, reqno]{smfart}

\usepackage[english]{babel}
\usepackage[T1]{fontenc}
\usepackage{cfr-lm}
\usepackage{smfthm}
\usepackage{amssymb}
\usepackage{mathtools}
\usepackage[mathscr]{euscript}
\usepackage{enumitem}
\usepackage{stmaryrd}
\usepackage{microtype}
\usepackage[dvipsnames]{xcolor}
\usepackage{hyperref}

\usepackage[backend=biber, maxnames=50]{biblatex}
\usepackage{csquotes}
\renewbibmacro{in:}{}
\bibliography{main.bib}

\author{Téofil Adamski}
\email{teofil.adamski@univ-smb.fr}
\address{Université Savoie Mont Blanc, CNRS, LAMA, 73000 Chambéry, France}
\title{On non-Archimedean and motivic distributions defined by kernels}

\usepackage{tikz-cd}
\tikzcdset{
    row sep/normal=1cm, column sep/normal=1cm,
    every label/.append style={font=\normalsize},
    outer xsep=2pt,
}

\hypersetup{
    colorlinks=true,
    linkcolor=Red,
    citecolor=Red,
    bookmarksdepth=3,
}

\renewcommand{\bfseries}{\sbweight}
\DeclareMathAlphabet{\mathbf}{\encodingdefault}{\familydefault}{sb}{n}

\newcommand{\RR}{\mathbf{R}}
\newcommand{\CC}{\mathbf{C}}
\newcommand{\QQ}{\mathbf{Q}}
\newcommand{\ZZ}{\mathbf{Z}}
\newcommand{\FF}{\mathbf{F}}
\newcommand{\cf}{\hbox{\normalfont\textbf{\textl{1}}}}

\DeclareMathOperator{\ord}{ord}
\DeclareMathOperator{\ac}{\overline{ac}}
\DeclareMathOperator{\supp}{supp}
\DeclareMathOperator{\WF}{WF}
\DeclareMathOperator{\id}{id}
\DeclareMathOperator{\Hom}{Hom}

\newcommand{\Def}{\mathbf{Def}}
\newcommand{\Field}{\mathbf{Field}}
\newcommand{\LL}{\mathbf{L}}
\renewcommand{\L}{\mathscr{L}}
\newcommand{\Dis}{\mathscr{D}}
\newcommand{\Sch}{\mathscr{S}}
\newcommand{\SC}{\mathscr{E}}
\newcommand{\Cons}{\mathscr{C}}
\newcommand{\Inte}{\mathscr{I}}
\newcommand{\class}{\mathscr{C}}
\newcommand{\Ball}{\mathrm{B}}
\newcommand{\expe}{^\text{exp}}
\newcommand{\Exp}{\mathbf{E}}
\newcommand{\Ker}{\mathscr{K}}
\newcommand{\n}{\mathfrak{n}}
\newcommand{\B}{\mathscr{B}}
\newcommand{\A}{\mathscr{A}}
\renewcommand{\r}{\mathfrak{r}}

\newcommand{\fnc}[4]{\left\lvert\begin{aligned}#1&\longrightarrow#2\\#3&\longmapsto#4\end{aligned}\right.\kern-\nulldelimiterspace}
\newcommand{\fonc}[5]{#1\colon\fnc{#2}{#3}{#4}{#5}}
\newcommand{\Sys}[1]{\left\{\begin{aligned}#1\end{aligned}\right.\kern-\nulldelimiterspace}
\newcommand{\dd}{\mathop{}\mathopen{}\mathrm{d}}
\renewcommand{\to}{\longrightarrow}

\newcommand{\inner}[2]{\langle#1, #2\rangle}
\newcommand{\points}[1]{\lvert#1\rvert}
\newcommand{\ls}[1]{(\!(#1)\!)}

\newcommand{\subproof}[1]{\smallbreak\textbf{$\diamond$ #1.}}
\newcommand{\sepproof}{\bigbreak\begin{center}$*~~*~~*$\end{center}\bigbreak}

\makeatletter
\SetSymbolFont{stmry}{bold}{U}{stmry}{m}{n}
\g@addto@macro\bfseries{\boldmath}

\def\th@plain{%
    \let\thm@indent\noindent
    \thm@headfont{\bfseries\smf@boldmath\itshape}%
    \thm@notefont{\bfseries\smf@boldmath\upshape}%
    \thm@preskip\bigskipamount%
    \thm@postskip\thm@preskip%
    \thm@headpunct{\MakePointrait}
    \itshape}

\begin{document}

\frontmatter

\begin{abstract}
    As in real microlocal analysis, we prove a Schwartz kernel theorem for~$p$-adic distributions. We extend this result for motivic distributions using Cluckers--Loeser's motivic integration. In both settings, we give also a relation between the wave front sets of the distribution and its kernel.
\end{abstract}

\maketitle

\tableofcontents

\mainmatter

The theory of real distributions has been introduced by Sobolev~\cite{Sobolev} and Schwartz~\cite{Schwartz}, it generalizes the notion of functions and it creates an adapted setting to find weak solutions to partial differential equations. In this context, to each distribution on a product space, we can associate a \emph{kernel}. That object was first introduced by Schwartz~\cite{Schwartz1950} and the \emph{Schwartz kernel theorem} can be found for instance in Hörmander's book~\cite{Hormander} (see Theorem~\ref{thm:kernel-real}).

A kernel is an object which generalizes integral operators, i.e. operators~$T$ on some space of the form
\[ Tf(y) = \int f(x)K(x, y) \dd x \]
for some function $K$ (for instance, the Fourier transform is an integral operator with~$K(x, y) = e^{-i\inner xy}$). In particular, kernels generalize differential operators (see~\cite[§0.4]{AlinhacGerard}) and pseudo-differential operators (see~\cite[§I.3]{AlinhacGerard}). This generalization can be done using the theory of distributions. As an application in the theory of linear partial differential equations, kernels can be used to compute inverses of differential operators (see~\cite[§52]{Treves} or~\cite[Theorem~13.2.1 \&{} Remark~p.~185]{HormanderII}).

Let us recall some notations. Let $m$ and $n$ be two positive integers. We consider the~$m$-dimensional euclidean space $\RR^m$. Let $X$ be an open set of $\RR^m$ and $Y$ be an open set of $\RR^n$. We denote by $\Dis(X)$ the set of test functions on $X$ ---~functions of class $\class^\infty$ with compact support on $X$~--- and~$\Dis'(X)$ the set of distributions on $X$ ---~continuous linear forms on $\Dis(X)$. For tests functions $\phi$ on $X$ and $\psi$ on~$Y$, we denote by $\phi \otimes \psi$ the test function on $X \times Y$ defined by the relation
\[ \phi \otimes \psi(x, y) = \phi(x)\psi(y), \qquad \forall (x, y) \in X \times Y. \]
We can now state the Schwartz kernel theorem.

\begin{theo}[Schwartz kernel theorem, Hörmander~{\cite[Theorem~5.2.1]{Hormander}}]\label{thm:kernel-real}
    For each distribution $u$ on the product $X \times Y$, there exists a unique continuous linear map $\Ker \colon \Dis(Y) \to \Dis'(X)$ such that
    \begin{equation}\label{eq:kernel}
        \inner{u}{\phi \otimes \psi} = \inner{\Ker\psi}{\phi}, \qquad \forall \phi \in \Dis(X), \; \forall \psi \in \Dis(Y).
    \end{equation}
    Conversely, for each continuous linear map $\Ker \colon \Dis(Y) \to \Dis'(X)$, there exists a unique distribution $u$ on the product $X \times Y$ such that equality~\eqref{eq:kernel} holds.
\end{theo}

Such a map $\Ker$ is called the \emph{kernel} of the distribution $u$. For a distribution $u$ on~$X \subset \RR^m$, we denote by $\WF(u) \subset X \times (\RR^m \setminus \{0\})$ its wave front set~\cite[§8.1]{Hormander}. We have the following relation between wave front sets.

\begin{theo}[Hörmander~{\cite[Theorem~8.2.12]{Hormander}}]\label{thm:WF-real}
    Let $u$ be a distribution on $X \times Y$ and $\Ker$ be its kernel. Let $\psi$ be a function of $\Dis(Y)$. Then
    \[ \WF(\Ker\psi) \subset \{(x, \xi) \in X \times (\RR^m \setminus \{0\}) \mid \exists y \in \supp\psi, \; ((x, y), (\xi, 0)) \in \WF(u)\}. \]
\end{theo}

In the $p$-adic setting, we also have a notion of distribution. Since the $p$-adic numbers field $\QQ_p$ is totally disconnected, a distribution is simply a linear form over the space of Schwartz--Bruhat functions (see~\cite[§7.1]{Igusa}). In this context, a notion of kernel can be formulated and, as in the real case, kernels generalize pseudo-differential operators (see~\cite{ZunigaGalindo}). A $p$-adic Schwartz kernel theorem exists as it is presented in the books~\cite[§4.6]{AlbeverioKhrennikovShelkovich} and~\cite[§VI.7]{Vladimirov}, but proofs are incomplete in both cases. In §\ref{sec:p}, we propose to complete these proofs (see Theorem~\ref{thm:kernel-p}). Moreover, as Heifetz does~\cite[Theorem~2.11]{Heifetz} without proof, we give an analogue to Theorem~\ref{thm:WF-real} (see Theorem~\ref{thm:WF-p}).

In §\ref{sec:mot}, we consider the case of the non-Archimedean field $k\ls t$ for a characteristic zero field $k$. In that setting, we need to use motivic integration introduced by Kontsevich, developed successively by Denef--Loeser~\cite{DenefLoeser1999}, Cluckers--Loeser~\cite{CluckersLoeser2008, CluckersLoeser2010} and Hrushovski--Kazhdan~\cite{HrushovskiKazhdan}. We will use Cluckers--Loeser's motivic integration. In this context, Raibaut~\cite{Raibaut} has developed motivic distributions. They satisfy some analogous properties to those verified by $p$-adic distributions. We prove a motivic Schwartz kernel theorem (see Theorem~\ref{thm:kernel-mot}) and a result of the motivic wave front set on the kernel (see Theorem~\ref{thm:WF-mot}), analogous to Theorem~\ref{thm:WF-p}.

\section{Kernels of \texorpdfstring{$p$-adic}{p-adic} distributions}
\label{sec:p}

Let $p$ be a prime number. To start, we recall some basics on $p$-adic distributions based on the books~\cite{AlbeverioKhrennikovShelkovich, Igusa}. We endowed the $p$-adic numbers field $\QQ_p$ with the $p$-adic valuation~$\ord \colon \QQ_p \to \ZZ \cup \{\infty\}$ and we consider its ring of integers $\ZZ_p$. With the valuation, the valued field $\QQ_p$ is a non-Archimedean metric space. We firstly define Schwartz--Bruhat functions and distributions. Let $m$ be a positive integer.

\begin{rema}
    As in~\cite{CluckersHalupczokLoeserRaibaut}, instead of working on the field $\QQ_p$, we can more generally work on a non-Archimedean local field, namely a finite field extension of $\QQ_p$ or a field of Laurent series with coefficients in a finite field. But for the convenience of the reader, in all the results and their proof, we will work on the field $\QQ_p$.
\end{rema}

\subsection{Distributions and Schwartz kernel theorem}

\begin{defi}
    A complex-valued function $\varphi$ on an open set $X$ of $\QQ_p^m$ is a \emph{Schwartz--Bruhat function} if it is locally constant (for the $p$-adic topology) and its support
    \[ \supp\varphi \coloneq \overline{\{x \in X \mid \varphi(x) \ne 0\}} \]
    is compact.
\end{defi}

We denote by $\Sch(X)$ the complex vector space of Schwartz--Bruhat functions on $X$%
    \footnote{The space $\Sch(X)$ can be equipped with a direct limit topology (see~\cite[pp. 98--99]{Igusa}) such that the algebraic dual of the space $\Sch(X)$ and its topological dual coincide. That is why, in Definition~\ref{def:distribution}, no continuity assumption is made. One can remarks that this topology corresponds analogously to the real case: the space $\Dis(X)$ for an open set $X$ of $\RR^m$ is a Fréchet space and its topology can be seen as a direct limit topology.}
and by $\class^\infty(X)$ the complex vector space of locally constant functions on $X$.

A \emph{polydisc} is a subset of $\QQ_p^m$ of the form $\Ball(x_0, \alpha) \coloneq x_0 + p^\alpha\ZZ_p^m$ for some vector~$x_0$ of $\QQ_p^m$ and some integer $\alpha$. Remark that a polydisc is not a singleton and hence open. Characteristic functions of bounded subsets of $\QQ_p^m$, like polydiscs, are examples of Schwartz--Bruhat functions on $\QQ_p^m$. We have the following representation theorem of Schwartz--Bruhat functions.

\begin{theo}\label{thm:dec-phi}
    Let $X$ be an open set of $\QQ_p^m$ and $\varphi$ be a Schwartz--Bruhat function on~$X$. Then there exist a positive integer $N$, complex numbers $c_i$ and polydiscs $B_i$ in~$X$ with~$1 \leqslant i \leqslant N$ such that
    \begin{equation}\label{eq:dec-phi}
        \varphi = \sum_{i = 1}^N c_i\cf_{B_i}.
    \end{equation}
    Moreover, the polydiscs $B_i$ with $1 \leqslant i \leqslant N$ can be chosen to be disjoint.
\end{theo}

\begin{proof}
    Since the function $\varphi$ is locally constant, for each point $x$ of $X$, there exists a polydisc $B_x$ in $X$ with center $x$ such that the function $\varphi$ is constant on $B_x$. Then we can write
    \[ \supp\varphi \subseteq \bigcup_{x \in \supp\varphi} B_x. \]
    Since the set $\supp\varphi$ is compact, we can find a positive integer $N$ and points $x_i$ of~$\supp\varphi$ with $1 \leqslant i \leqslant N$ such that
    \begin{equation}\label{eq:dec-supp}
        \supp\varphi \subseteq \bigcup_{i = 1}^N B_{x_i}.
    \end{equation}
    Let us show that we can assume that the polydiscs $B_i \coloneq B_{x_i}$ with $1 \leqslant i \leqslant N$ are disjoints. Let $i$ and $j$ be two indexes between $1$ and $N$. If $B_i \cap B_j = \emptyset$, then there is nothing to do. We assume now that $B_i \cap B_j \ne \emptyset$. Since the field $\QQ_p$ is non-Archimedean, one of the polydisc $B_i$ or $B_j$ is included on the other and we delete the smallest. Repeating this procedure, we can assume that the polydiscs $B_i \coloneq B_{x_i}$ with $1 \leqslant i \leqslant N$ are disjoints.

    By setting $c_i \coloneq \varphi(x_i)$ for $1 \leqslant i \leqslant N$, by the assumption made in the last paragraph, equality~\eqref{eq:dec-phi} follows from equation~\eqref{eq:dec-supp}.
\end{proof}

We define what a $p$-adic distribution is. Let $X$ be an open set of $\QQ_p^m$.

\begin{defi}\label{def:distribution}
    A \emph{distribution} on the open set $X$ is a linear form on the complex vector space $\Sch(X)$.
\end{defi}

We denote by $\Sch'(X)$ the complex vector space of distributions on $X$. From now on, we consider the Haar measure on $\QQ_p$ such that the volume of $\ZZ_p$ is $1$.

\begin{defi}
    For a locally constant function $f$ on $X$, we consider the distribution~$T_f$ on $X$ defined by the equality
    \[ \inner{T_f}{\varphi} = \int_X f(x)\varphi(x) \dd x \]
    for all Schwartz--Bruhat functions $\varphi$ on $X$.
\end{defi}

\begin{prop}\label{prop:injection-Cinf}
    The $\CC$-linear map
    \[ \fonc{T}{\class^\infty(X)}{\Sch'(X),}{f}{T_f} \]
    is injective.
\end{prop}

\begin{proof}
    Let $f$ be a function in the kernel of the $\CC$-linear map $T$. In other words, we have
    \begin{equation}\label{eq:kernel-injection}
        \forall \varphi \in \Sch(X), \qquad \int_X f(x)\varphi(x) \dd x = 0.
    \end{equation}
    Let $K$ be a nonempty compact open set of $X$ (a polydisc typically). We can decompose the function $f\cf_K$ using Theorem~\ref{thm:dec-phi}: we write it as
    \[ f\cf_K = \sum_{i = 1}^N c_i\cf_{B_i} \]
    for some nonnegative integer $N$, some complex numbers $c_i$ and some polydiscs $B_i$ of~$X$ with $1 \leqslant i \leqslant N$. For each index $i$ in~$\{1, \dots, N\}$, we can apply relation~\eqref{eq:kernel-injection} with~$\varphi = \cf_{B_i}\cf_K$ which leads to $c_i = 0$. Thus, the function $f$ is zero on $K$. We deduce that the function $f$ is zero on $X$.
\end{proof}

\begin{defi}
    A distribution $u$ on $X$ is \emph{represented by a $\class^\infty$-function} if there exists a locally constant function $f \colon X \to \CC$ such that $u = T_f$.
\end{defi}

For a locally constant function $f \colon X \to \CC$, we denote again by $f$ the distribution~$T_f$. This notation is legitimate thanks to Proposition~\ref{prop:injection-Cinf}.

\bigbreak

Let $X$ be an open set of $\QQ_p^m$ and $Y$ be an open set of $\QQ_p^n$. As in the real setting, for two Schwartz--Bruhat functions $\phi$ on $X$ and $\psi$ on $Y$, we consider the Schwartz--Bruhat function on $X \times Y$
\[ \fonc{\phi \otimes \psi}{X \times Y}{\CC,}{(x, y)}{\phi(x)\psi(y).} \]
In this way, we obtain a bilinear map
\[ - \otimes - \colon \Sch(X) \times \Sch(Y) \to \Sch(X \times Y) \]
which can be factorised into a $\CC$-linear map
\begin{equation}\label{eq:iso-Sch}
    - \otimes - \colon \Sch(X) \otimes_\CC \Sch(Y) \to \Sch(X \times Y)
\end{equation}
We prove in Lemma~\ref{lem:surj} that this latter map is surjective. We can also prove that this map is injective ---~it is a general fact about tensor products of complex valued functions.

\begin{lemm}\label{lem:surj}
    The map~\eqref{eq:iso-Sch} is surjective.
\end{lemm}

\begin{proof}
    Let $\varphi$ be a Schwartz--Bruhat function on $X \times Y$. According to Theorem~\ref{thm:dec-phi}, there exist a positive integer $N$, complex numbers $c_i$ and disjoints polydiscs $B_i$ in~$X \times Y$ with~$1 \leqslant i \leqslant N$ such that
    \begin{equation}\label{eq:dec-phi-2}
        \varphi = \sum_{i = 1}^N c_i\cf_{B_i}.
    \end{equation}
    Each polydisc $B_i$ with $1 \leqslant i \leqslant N$ can be written as $B_i = C_i \times D_i$ for two polydiscs $C_i$ in $X$ and $D_i$ in $Y$. Thus, equality~\eqref{eq:dec-phi-2} can be rewritten as
    \[ \varphi = \sum_{i = 1}^N c_i\cf_{C_i} \otimes \cf_{D_i} \]
    which concludes the lemma.
\end{proof}

We now prove the following $p$-adic Schwartz kernel theorem.

\begin{theo}\label{thm:kernel-p}
    Let $X$ be an open set of $\QQ_p^m$ and $Y$ be an open set of $\QQ_p^n$. For each distribution $u$ on $X \times Y$, there exists a unique $\CC$-linear map
    \[ \Ker \colon \Sch(Y) \to \Sch'(X) \]
    such that
    \begin{equation}\label{eq:kernel-p}
        \inner{u}{\phi \otimes \psi} = \inner{\Ker\psi}{\phi}, \qquad \forall \phi \in \Sch(X), \; \forall \psi \in \Sch(Y).
    \end{equation}
    Conversely, for each $\CC$-linear map $\Ker \colon \Sch(Y) \to \Sch'(X)$, there exists a unique distribution $u$ on $X \times Y$ such that equality~\eqref{eq:kernel-p} holds.
\end{theo}

\begin{defi}
    Such a map $\Ker$ is called the \emph{kernel} of the distribution $u$.
\end{defi}

\begin{exem}\label{exem:locally-integrable-kernel}
    Let $f$ be a locally constant function on $X \times Y$. It defines a distribution on $X \times Y$. Let $\Ker$ be its kernel and $\psi$ be a Schwartz--Bruhat function on~$Y$. We can check that, by Theorem~\ref{thm:dec-phi}, the function
    \[ \fnc{X}{\CC,}{x}{\int_Y f(x, y)\psi(y) \dd y} \]
    is locally constant and, by Fubini's theorem, it represents the distribution $\Ker\psi$.
\end{exem}

\begin{exem}
    Let $X$ be an open set of $\QQ_p^m$. By Proposition~\ref{prop:injection-Cinf}, we have a $\CC$-linear injection $\class^\infty(X) \to \Sch'(X)$. If we pre-compose it with the inclusion~$\Sch(X) \to \class^\infty(X)$, we obtain a $\CC$-linear map
    \[ \fnc{\Sch(X)}{\Sch'(X),}{\psi}{T_\psi.} \]
    It is the kernel associated to the distribution
    \[ u_\Delta \colon \varphi \longmapsto \int_X \varphi(x, x) \dd x \]
    on the product $X \times X$, induced by the diagonal $\Delta \coloneq \{(x, y) \in X \times X \mid x = y\}$. Indeed, for all Schwartz--Bruhat functions~$\phi$ and $\psi$ on $X$, we have
    \begin{align*}
        \inner{T_\psi}{\phi} &= \int_X \psi(x)\phi(x) \dd x \\
        &= \int_X \phi \otimes \psi(x, x) \dd x \\
        &= \inner{u_\Delta}{\phi \otimes \psi}.
    \end{align*}
\end{exem}

\begin{proof}[Proof of Theorem~\ref{thm:kernel-p}]
    Let $u$ be a distribution on $X \times Y$. Equality~\eqref{eq:kernel-p} defines a unique such map~$\Ker$ and it is $\CC$-linear because of the bilinearity of the tensor product and the linearity of the distribution~$u$.

    \sepproof

    Let us prove the reciprocal. Let $\Ker \colon \Sch(Y) \to \Sch'(X)$ be a $\CC$-linear map.

    \subproof{Existence}
    Let $\varphi$ be a Schwartz--Bruhat function on $X \times Y$. According to Theorem~\ref{thm:dec-phi} and its proof, there exist a positive integer $N$, complex numbers $c_i$ with~$1 \leqslant i \leqslant N$, disjoints polydiscs $C_i$ in $X$ with~$1 \leqslant i \leqslant N$ and disjoints polydiscs~$D_i$ in $Y$ with~$1 \leqslant i \leqslant N$ such that
    \begin{equation}\label{eq:dec-phi-3}
        \varphi = \sum_{i = 1}^N c_i\cf_{C_i} \otimes \cf_{D_i}.
    \end{equation}
    We set now
    \begin{equation}\label{eq:def-u-p}
        \inner u\varphi \coloneq \sum_{i = 1}^N c_i\inner{\Ker\cf_{D_i}}{\cf_{C_i}}.
    \end{equation}
    With the next paragraph, this definition does not depend of the choice in decomposition~\eqref{eq:dec-phi-3}. Moreover, we can easily verify that the map $u \colon \Sch(X) \to \CC$ is $\CC$-linear. Consequently, we obtain a distribution $u$ on $X \times Y$ and equality~\eqref{eq:kernel-p} holds.

    We just have to prove that definition~\eqref{eq:def-u-p} does not depend of the choice of objects involved. We consider two decompositions of the function $\varphi$: for $k = 1$ or~$k = 2$, we write it as
    \begin{equation}\label{eq:dec-phi-k}
        \varphi = \sum_{i = 1}^{N_k} c_{k, i}\cf_{\Ball(a_{k, i}, \alpha_{k, i})} \otimes \cf_{\Ball(b_{k, i}, \beta_{k, i})}
    \end{equation}
    for some nonnegative integer $N_k$%
        \footnote{With the usual convention, if $N_k = 0$, then the sum is zero.}, complex numbers $c_{k, i}$, vectors $a_{k, i}$ of $\QQ_p^m$ and vectors~$b_{k, i}$ of $\QQ_p^n$ with $1 \leqslant i \leqslant N_k$. We want to prove the equality
    \begin{equation}\label{eq:independance}
        \sum_{i = 1}^{N_1} c_{1, i}\inner{\Ker\cf_{\Ball(b_{1, i}, \beta_{1, i})}}{\cf_{\Ball(a_{1, i}, \alpha_{1, i})}} = \sum_{i = 1}^{N_2} c_{2, i}\inner{\Ker\cf_{\Ball(b_{2, i}, \beta_{2, i})}}{\cf_{\Ball(a_{2, i}, \alpha_{2, i})}}.
    \end{equation}
    We can assume that complex numbers $c_{k, i}$ with $k = 1, 2$ and $1 \leqslant i \leqslant N_k$ are all nonzero ---~otherwise, we just have to remove the corresponding terms. We set
    \[ \gamma \coloneq \max_{\substack{k = 1, 2 \\ 1 \leqslant i \leqslant N_k}} (\alpha_{i, k}, \beta_{i, k}). \]
    Let $k = 1$ or $k = 2$. Decomposition~\eqref{eq:dec-phi-k} can be refined as
    \begin{multline}\label{eq:dec-phi-k-refined}
        \varphi = \\
        \sum_{i = 1}^{N_k} c_{k, i} \sum_{\substack{x_{\alpha_{k, i}}, \dots, x_\gamma \in \FF_p^m \\ y_{\beta_{k, i}}, \dots, y_\gamma \in \FF_p^n}} \cf_{\Ball(a_{k, i} + p^{\alpha_{k, i}}x_{\alpha_{k, i}} + \dots + p^\gamma x_\gamma, \gamma + 1)} \otimes \cf_{\Ball(b_{k, i} + p^{\beta_{k, i}}y_{\beta_{k, i}} + \dots + p^\gamma y_\gamma, \gamma + 1)}
    \end{multline}
    where the notation $\FF_p$ stands for the residue field of the valued field $\QQ_p$.
    Thanks to linearity, we can write
    \begin{multline}\label{eq:independance-refined}
        \sum_{i = 1}^{N_k} c_{k, i}\inner{\Ker\cf_{\Ball(b_{k, i}, \beta_{k, i})}}{\cf_{\Ball(a_{k, i}, \alpha_{k, i})}} \\
        = \sum_{i = 1}^{N_k} c_{k, i} \sum_{\substack{x_{\alpha_{k, i}}, \dots, x_\gamma \in \FF_p^m \\ y_{\beta_{k, i}}, \dots, y_\gamma \in \FF_p^n}} \inner{\Ker\cf_{\Ball(b_{k, i} + p^{\beta_{k, i}}y_{\beta_{k, i}} + \dots + p^\gamma y_\gamma, \gamma + 1)}}{ \\
        \cf_{\Ball(a_{k, i} + p^{\alpha_{k, i}}x_{\alpha_{k, i}} + \dots + p^\gamma x_\gamma, \gamma + 1)}}.
    \end{multline}
    To prove equality~\eqref{eq:independance}, thanks to last equality~\eqref{eq:independance-refined}, it is enough to prove that decompositions~\eqref{eq:dec-phi-k-refined} for $k = 1$ and $k = 2$ are the same up to a permutation. We remark that polydiscs
    \[ \Ball(a_{k, i} + p^{\alpha_{k, i}}x_{\alpha_{k, i}} + \dots + p^\gamma x_\gamma, \gamma + 1) \]
    with $1 \leqslant i \leqslant N_k$ and $x_{\alpha_{k, i}}, \dots, x_\gamma \in \FF_p^n$ are disjoints and polydiscs
    \[ \Ball(b_{k, i} + p^{\beta_{k, i}}y_{\beta_{k, i}} + \dots + p^\gamma y_\gamma, \gamma + 1) \]
    with $1 \leqslant i \leqslant N_k$ and $y_{\alpha_{k, i}}, \dots, y_\gamma \in \FF_p^n$ are disjoints. Let $i$ be an index between $1$ and~$N_1$ and $x_{\alpha_{k, i}}$, \dots, $x_\gamma$ vectors of $\FF_p^m$. We can find an index $j_i$ between $1$ and $N_2$ and vectors $y_{\alpha_{k, i}}$, \dots, $y_\gamma$ of $\FF_p^n$ such that
    \[ a_{1, i} + p^{\alpha_{1, i}}x_{\alpha_{1, i}} + \dots + p^\gamma x_\gamma \in \Ball(a_{2, j_i} + p^{\alpha_{2, j_i}}x_{\alpha_{2, j_i}} + \dots + p^\gamma x_\gamma, \gamma + 1) \]
    and
    \[ b_{1, i} + p^{\beta_{1, i}}y_{\beta_{1, i}} + \dots + p^\gamma y_\gamma \in \Ball(b_{2, j_i} + p^{\beta_{2, j_i}}y_{\beta_{2, j_i}} + \dots + p^\gamma y_\gamma, \gamma + 1). \]
    Otherwise, we would have
    \[ \varphi(a_{1, i} + p^{\alpha_{1, i}}x_{\alpha_{1, i}} + \dots + p^\gamma x_\gamma, b_{1, i} + p^{\beta_{1, i}}y_{\beta_{1, i}} + \dots + p^\gamma y_\gamma) = 0 \]
    and then $c_i = 0$ which is impossible. We deduce that $c_{1, i} = c_{2, j_i}$ and then
    \[ \Ball(a_{1, i} + p^{\alpha_{1, i}}x_{\alpha_{1, i}} + \dots + p^\gamma x_\gamma, \gamma + 1) = \Ball(a_{2, j_i} + p^{\alpha_{2, j_i}}x_{\alpha_{2, j_i}} + \dots + p^\gamma x_\gamma, \gamma + 1) \]
    and
    \[ \Ball(b_{1, i} + p^{\beta_{1, i}}y_{\beta_{1, i}} + \dots + p^\gamma y_\gamma, \gamma + 1) = \Ball(b_{2, j_i} + p^{\beta_{2, j_i}}y_{\beta_{2, j_i}} + \dots + p^\gamma y_\gamma, \gamma + 1). \]
    Finally, we obtain the equality
    \begin{multline*}
        c_{1, i}\cf_{\Ball(a_{1, i} + p^{\alpha_{1, i}}x_{\alpha_{1, i}} + \dots + p^\gamma x_\gamma, \gamma + 1)} \otimes \cf_{\Ball(b_{1, i} + p^{\beta_{1, i}}y_{\beta_{1, i}} + \dots + p^\gamma y_\gamma, \gamma + 1)} \\
        = c_{2, j_i}\cf_{\Ball(a_{1, j_i} + p^{\alpha_{1, j_i}}x_{\alpha_{1, j_i}} + \dots + p^\gamma x_\gamma, \gamma + 1)} \otimes \cf_{\Ball(b_{1, j_i} + p^{\beta_{1, j_i}}y_{\beta_{1, j_i}} + \dots + p^\gamma y_\gamma, \gamma + 1)}.
    \end{multline*}
    We can now delete a term in decompositions~\eqref{eq:dec-phi-k-refined}. Repeating this procedure, we deduce that decompositions~\eqref{eq:dec-phi-k-refined} are the same up to a permutation. This proves equality~\eqref{eq:independance} which concludes the existence.

    \subproof{Uniqueness}
    Let $u$ and $u'$ be two distributions on $X \times Y$ such that
    \begin{align}
        \inner{u}{\phi \otimes \psi} &= \inner{\Ker\psi}{\phi} \qquad \text{and} \label{eq:noyau-mot-u-p} \\
        \inner{u'}{\phi \otimes \psi} &= \inner{\Ker\psi}{\phi} \label{eq:noyau-mot-u'-p}
    \end{align}
    for all Schwartz--Bruhat functions $\phi$ on $X$ and $\psi$ on $Y$. Let $\varphi$ be a Schwartz--Bruhat function on $X \times Y$. We write it as decomposition~\eqref{eq:dec-phi-3}. With equalities~\eqref{eq:noyau-mot-u-p} and~\eqref{eq:noyau-mot-u'-p}, we deduce that
    \[ \inner u\varphi = \sum_{i = 1}^N c_i\inner{u}{\cf_{C_i} \otimes \cf_{D_i}} = \sum_{i = 1}^N c_i\inner{\Ker\cf_{D_i}}{\cf_{C_i}} = \sum_{i = 1}^N c_i\inner{u'}{\cf_{C_i} \otimes \cf_{D_i}} = \inner{u'}{\varphi}. \]
    It implies that $u = u'$.
\end{proof}

In other words, Theorem~\ref{thm:kernel-p} can be slightly refined as follows.

\begin{coro}
    Let $X$ be an open set of $\QQ_p^m$ and $Y$ be an open set of $\QQ_p^n$. Then there exists a canonical $\CC$-linear isomorphism
    \begin{equation}\label{eq:iso-Hom}
        \Sch'(X \times Y) \cong \Hom_\CC(\Sch(Y), \Sch'(X))
    \end{equation}
    where the notation $\Hom_\CC(\Sch(Y), \Sch'(X))$ stands for the complex vector space of linear maps from $\Sch(Y)$ to $\Sch'(X)$.
\end{coro}

\begin{proof}
    By bilinearity of the tensor product and the pairing between distributions and Schwartz--Bruhat functions, the map~$u \longmapsto \Ker$ given by Theorem~\ref{thm:kernel-p}, namely the map~\eqref{eq:iso-Hom}, is a $\CC$-linear isomorphism.
\end{proof}

\begin{rema}\label{rema:kernel-locally-constant}
    Isomorphism~\eqref{eq:iso-Hom} can be recovered using isomorphism~\eqref{eq:iso-Sch} and the adjunction between the endofunctors $- \otimes_\CC \Sch(Y)$ and $\Hom_\CC(\Sch(Y), -)$ on the category of complex vector spaces. Indeed, we get
    \begin{multline*}
        \Sch'(X \times Y) = \Hom_\CC(\Sch(X \times Y), \CC) \cong \Hom_\CC(\Sch(X) \otimes_\CC \Sch(Y), \CC) \\
        \cong \Hom_\CC(\Sch(Y), \Hom_\CC(\Sch(X), \CC)) = \Hom_\CC(\Sch(Y), \Sch'(X)).
    \end{multline*}
\end{rema}

\subsection{Wave front sets and kernels}

We use the definition of wave front sets for~$p$-adic distributions introduced by Heifetz~\cite{Heifetz} and used by Aizenbud--Drinfeld~\cite{AizenbudDrinfeld} and by Cluckers--Halupczok--Loeser--Raibaut~\cite{CluckersHalupczokLoeserRaibaut}. Let us recall it. Let $\Psi \colon \QQ_p \to \CC^\times$ be an additive character which is nontrivial on~$\ZZ_p$ and trivial on $p\ZZ_p$. Let $\Lambda$ be an open subgroup of finite index of~$\QQ_p^\times$. On the $p$-adic vector space $\QQ_p^m$, we consider the canonical inner product $\inner\cdot\cdot$. Let $X$ be an open set of $\QQ_p^m$.

\begin{defi}[Heifetz~{\cite[§2]{Heifetz}}, CHLR~{\cite[Definition~2.8.1 \&{} Theorem~2.5.2]{CluckersHalupczokLoeserRaibaut}}]
    A distribution $u$ on $X$ is \emph{$\Lambda$-microlocally smooth} at a point $(x_0, \xi_0)$ of~$X \times (\QQ_p^m \setminus \{0\})$ if there exist open neighborhoods $U$ of $x_0$ in $X$ and $\smash{\check U}$ of $\xi_0$ in~$\QQ_p^m \setminus \{0\}$ such that, for any Schwartz--Bruhat function $\varphi$ on $X$ such that $\supp\varphi \subset U$, there exists an integer~$N$ such that
    \begin{equation}\label{eq:WF-p}
        \inner{u}{\varphi\Psi(\inner{\cdot}{\lambda\xi})} = 0
    \end{equation}
    for all elements $\lambda$ of $\Lambda$ such that $\ord\lambda < N$ and all points $\xi$ of $\check U$.
\end{defi}

\begin{rema}
    According to~\cite[Theorem~2.5.2]{CluckersHalupczokLoeserRaibaut}, the function $\xi \longmapsto \inner{u}{\varphi\Psi(\inner\cdot\xi)}$ represents the Fourier transform $\widehat{\varphi u}$ of the distribution $\varphi u$. Consequently, formula~\eqref{eq:WF-p} can be rewritten as
    \[ \widehat{\varphi u}(\lambda\xi) = 0. \]
\end{rema}

\begin{defi}\label{def:WF-p}
    The complement of the set of points of $X \times (\QQ_p^m \setminus \{0\})$ at which a distribution $u$ on $X$ is $\Lambda$-microlocally smooth is the \emph{$\Lambda$-wave front set} of the distribution~$u$, denoted by $\WF_\Lambda(u)$. We also set $\WF_\Lambda^0(u) \coloneq \WF_\Lambda(u) \cup (X \times \{0\})$.
\end{defi}

We give an analogue of Theorem~\ref{thm:WF-real} in the $p$-adic setting. Let $X$ be an open set of~$\QQ_p^m$ and $Y$ be an open set of $\QQ_p^n$.

\begin{theo}\label{thm:WF-p}
    Let $u$ be a distribution on $X \times Y$ and $\Ker$ be its kernel. Let $\psi$ be a Schwartz--Bruhat function on $Y$. Then
    \[ \WF_\Lambda(\Ker\psi) \subset \{(x, \xi) \in X \times (\QQ_p^m \setminus \{0\}) \mid \exists y \in \supp\psi, \; ((x, y), (\xi, 0)) \in \WF_\Lambda(u)\}. \]
\end{theo}

\begin{proof}
    We will the prove the complementary inclusion of the desired one. Let~$(x_0, \xi_0)$ be a point of $X \times (\QQ_p^m \setminus \{0\})$ such that
    \begin{equation}\label{eq:assumption-x0-xi0}
        \forall y \in \supp\psi, \qquad ((x_0, y), (\xi_0, 0)) \notin \WF_\Lambda(u).
    \end{equation}
    We want to show that $(x_0, \xi_0) \notin \WF_\Lambda(\Ker\psi)$. Let $y$ be a point of $\supp\psi$. By assumption~\eqref{eq:assumption-x0-xi0}, we can find some open neighborhoods $U_y$ of $(x_0, y)$ in $X \times Y$ and~$\check U_y$ of~$(\xi_0, 0)$ in $\QQ_p^{m + n} \setminus \{0\}$ such that, for any Schwartz--Bruhat function $\varphi$ on $X \times Y$ such that~$\supp\varphi \subset U_y$, there exists an integer~$N_y(\varphi)$ such that
    \begin{equation}\label{eq:WF-u}
        \inner{u}{\varphi\Psi(\inner{\cdot}{\lambda(\xi, \eta)})} = 0
    \end{equation}
    for all elements $\lambda$ of $\Lambda$ such that $\ord\lambda < N_y(\varphi)$ and all points $(\xi, \eta)$ of $\check U_y$. Reducing open sets $U_y$ and $\check U_y$, we can assume that the latter are polydiscs: we can write it as
    \[ U_y = B_{X, y} \times B_{Y, y} \qquad \text{and} \qquad
       \check U_y = \check B_{X, y} \times \check B_{Y, y} \]
    for some polydiscs $B_{X, y}$ of $X$, $B_{Y, y}$ of $Y$, $\check B_{X, y}$ of $\QQ_p^m \setminus \{0\}$ and $\check B_{Y, y}$ of $\QQ_p^n$. Moreover, we have
    \[ \supp\psi \subset \bigcup_{y \in \supp\psi} B_{Y, y}. \]
    Since the subset $\supp\psi$ is compact, we can find a finite set $I$ of $\supp\psi$ such that
    \begin{equation}\label{eq:supp-psi-compact}
        \supp\psi \subset \bigcup_{y \in I} B_{Y, y}.
    \end{equation}
    By proceeding as in the proof of Theorem~\ref{thm:dec-phi}, we can assume that the polydiscs $U_y$ with $y \in I$ are disjoint. Let us consider respectively these two open neighborhoods
    \[ V \coloneq \bigcap_{y \in I} B_{X, y} \subset X \qquad \text{and} \qquad
       \check V \coloneq \bigcap_{y \in I} \check B_{X, y} \subset \QQ_p^m \setminus \{0\} \]
    of the points $x_0$ and $\xi_0$. With them, we show that $(x_0, \xi_0) \notin \WF_\Lambda(\Ker\psi)$. Let $\phi$ be a Schwartz--Bruhat function on $X$ such that $\supp\phi \subset V$. We consider the Schwartz--Bruhat function $\varphi \coloneq \phi \otimes \psi$ on $X \times Y$ whose support verifies inclusions
    \[ \supp\varphi \subset \supp\phi \times \supp\psi \subset V \times \bigcup_{y \in I} B_{Y, y} \subset \bigsqcup_{y \in I} U_y \eqcolon U \]
    by equality~\eqref{eq:supp-psi-compact}. Furthermore, since the polydisc $U_y$ for $y \in I$ are disjoint, we write
    \begin{equation}\label{eq:dec-phi-Uy}
        \varphi = \sum_{y \in I} \varphi\cf_{U_y}.
    \end{equation}
    Let $\lambda$ be an element of $\Lambda$ such that
    \[ \ord\lambda < \min_{y \in I} N_y(\varphi\cf_{U_y}) \]
    and $\xi$ be an element of $\check V$. For each point $y$ of $I$, we find these three assertions
    \[ \supp\varphi\cf_{U_y} \subset U_y, \qquad
       (\xi, 0) \in \check U_y \qquad \text{and} \qquad
       \ord\lambda < N_y(\varphi\cf_{U_y}) \]
    which, by equation~\eqref{eq:WF-u}, lead to the equality
    \begin{equation}\label{eq:phi-Uy-zero}
        \inner{u}{\varphi\cf_{U_y}\Psi(\inner{\cdot}{\lambda(\xi, 0)})} = 0.
    \end{equation}
    Combining equations~\eqref{eq:dec-phi-Uy} and~\eqref{eq:phi-Uy-zero}, we obtain
    \begin{align*}
        \inner{\Ker\psi}{\phi\Psi(\inner{\cdot}{\lambda\xi})} &= \inner{u}{\phi\Psi(\inner{\cdot}{\lambda\xi}) \otimes \psi} \\
        &= \inner{u}{\varphi\Psi(\inner{\cdot}{\lambda(\xi, 0)})} \\
        &= \sum_{y \in I} \inner{u}{\varphi\cf_{U_y}\Psi(\inner{\cdot}{\lambda(\xi, 0)})} = 0
    \end{align*}
    which concludes that $(x_0, \xi_0) \notin \WF_\Lambda(\Ker\psi)$.
\end{proof}

We now give the analog result of~\cite[Theorem~8.2.13]{Hormander}. Until the end of this section, we consider a distribution $u$ on $X \times Y$ and its associated kernel $\Ker$. We set
\begin{align*}
    \WF'_\Lambda(u) &\coloneq \{(x, y, \xi, \eta) \in X \times Y \times (\QQ_p^{m + n} \setminus \{0\}) \mid (x, y, \xi, -\eta) \in \WF_\Lambda(u)\}, \\
    \WF_\Lambda(u)_Y &\coloneq \{(y, \eta) \in Y \times (\QQ_p^n \setminus \{0\}) \mid \exists x \in X, (x, y, 0, \eta) \in \WF_\Lambda(u)\} \qquad \text{and} \\
    \WF_\Lambda(u)_X &\coloneq \{(x, \xi) \in X \times (\QQ_p^m \setminus \{0\}) \mid \exists y \in Y, (x, y, \xi, 0) \in \WF_\Lambda(u)\}.
\end{align*}
We will use the notion of support of a distribution as defined in~\cite[Definition~2.4.3]{CluckersHalupczokLoeserRaibaut}, denoted by $\supp(-)$. We denote by $\SC'(Y)$ the complex vector space of distributions on~$Y$ with compact support and, for each compact subset $M$ of $Y$, by $\SC'_M(Y)$ the complex vector space of distributions on~$Y$ whose support is contained in $M$. Also for a closed $\Lambda$-cone $\Gamma$ of $Y \times (\QQ_p^n \setminus \{0\})$, we denote by $\Sch'_{\Gamma, \Lambda}(Y)$ the set of distributions on~$Y$ whose $\Lambda$-wave front set is contained in $\Gamma$ and we endow it with the topology given by~\cite[Definition~2.9.1]{CluckersHalupczokLoeserRaibaut} as well as the set $\Sch'(X)$. Finally, we use the notion of tensor product and product of distributions recalled in~\cite[Theorems~2.7.1 \&{}~2.9.8]{CluckersHalupczokLoeserRaibaut}.

\begin{lemm}
    The set $\Sch'_{\Gamma, \Lambda}(Y)$ is a complex vector subspace of $\Sch'(Y)$.
\end{lemm}

\begin{proof}
    From Definition~\ref{def:WF-p} of $\Lambda$-wave front set, for two distributions $u$ and $v$ on~$Y$ and two complex numbers $\alpha$ and $\beta$, we can show the inclusion
    \[ \WF_\Lambda(\alpha u + \beta v) \subset \WF_\Lambda(u) \cup \WF_\Lambda(v). \]
    Thus the lemma follows.
\end{proof}

\begin{theo}\label{thm:extension}
    Let $\Gamma$ be a closed $\Lambda$-cone of~$Y \times (\QQ_p^n \setminus \{0\})$ such that
    \begin{equation}\label{eq:Gamma}
        \Gamma \cap \WF'_\Lambda(u)_Y = \emptyset.
    \end{equation}
    Then the kernel $\Ker \colon \Sch(Y) \to \Sch'(X)$ of the distribution $u$ extends to a unique $\CC$-linear map
    \[ \Ker \colon \SC'(Y) \cap \Sch'_{\Gamma, \Lambda}(Y) \to \Sch'(X) \]
    such that, for all polydiscs $B$ of $Y$, its restriction
    \[ \Ker \colon \SC'_B(Y) \cap \Sch'_{\Gamma, \Lambda}(Y) \to \Sch'(X) \]
    is continuous\footnote{The topology of the domain is the topology induced by the topology on the space $\Sch'_{\Gamma, \Lambda}(Y)$.}. Moreover, let $v$ be a distribution of $\SC'(Y) \cap \Sch'_{\Gamma, \Lambda}(Y)$. Then
    \begin{equation}\label{eq:WF-Kv}
        \WF_\Lambda(\Ker v) \subset \WF_\Lambda(u)_X \cup \WF'_\Lambda(u) \circ \WF_\Lambda(v)
    \end{equation}
    where
    \begin{multline*}
        \WF'_\Lambda(u) \circ \WF_\Lambda(v) \coloneq \{(x, \xi) \in X \times (\QQ_p^m \setminus \{0\}) \mid \\
        \exists (y, \eta) \in \WF_\Lambda(v), (x, y, \xi, \eta) \in \WF'_\Lambda(u)\}.
    \end{multline*}
\end{theo}

\begin{rema}
    Let $\Gamma'$ be another $\Lambda$-cone of~$Y \times (\QQ_p^n \setminus \{0\})$ which satisfies equation~\eqref{eq:Gamma} such that $\Gamma \subset \Gamma'$. Then the two extensions
    \[ \Ker \colon \SC'(Y) \cap \Sch'_{\Gamma, \Lambda}(Y) \to \Sch'(X) \qquad \text{and} \qquad
       \Ker \colon \SC'(Y) \cap \Sch'_{\Gamma', \Lambda}(Y) \to \Sch'(X) \]
    coincide on $\SC'(Y) \cap \Sch'_{\Gamma, \Lambda}(Y)$.
\end{rema}

\begin{rema}
    If $\WF_\Lambda(u) \subset X \times Y \times (\QQ_p^m \setminus \{0\}) \times (\QQ_p^n \setminus \{0\})$, then $\WF_\Lambda(u)_X = \emptyset$ and equation~\eqref{eq:WF-Kv} becomes
    \[ \WF_\Lambda(\Ker v) \subset \WF'_\Lambda(u) \circ \WF_\Lambda(v). \]
    In that case, if furthermore $X = Y$, then it is usually said that the $\Lambda$-wave front set~$\WF_\Lambda(\Ker v)$ of the distribution $\Ker v$ is included in the~$\Lambda$-wave front set $\WF_\Lambda(v)$ of the distribution $v$ transported by the set $\WF'_\Lambda(u)$ as in~\cite[§II.B.2.1, Théorème~2.1]{AlinhacGerard}.
\end{rema}

\begin{proof}
    \subproof{Existence}
    Let $v$ be a distribution of $\SC'(Y) \cap \Sch'_{\Gamma, \Lambda}(Y)$. First, let us prove that the product of distributions $(1 \otimes v)u$ is well-defined. By~\cite[Theorem~2.8.10]{CluckersHalupczokLoeserRaibaut} for instance, we get
    \begin{align}
        \WF_\Lambda(1 \otimes v) &\subset \WF_\Lambda^0(1) \times \WF_\Lambda^0(v) \notag \\
        &= [\emptyset \cup (X \times \{0\})] \times \WF_\Lambda^0(v) \notag \\
        &= \{(x, y, \xi, \eta) \in X \times Y \times \QQ_p^{m + n} \mid \xi = 0, (y, \eta) \in \WF_\Lambda^0(v)\} \label{eq:inc-WF}
    \end{align}
    where Cartesian products are seen as products in the cotangent bundle of $X \times Y$ and we recall that $\WF_\Lambda^0(v) \coloneq \WF_\Lambda(v) \cup (Y \times \{0\})$ as in Definition~\ref{def:WF-p}. Then let us prove that
    \begin{equation}\label{eq:cond-produit}
        \forall (x, y, \xi, \eta) \in \WF_\Lambda(u), \qquad (x, y, -\xi, -\eta) \notin \WF_\Lambda(1 \otimes v).
    \end{equation}
    By contradiction, we assume that there exists an element~$(x, y, \xi, \eta)$ of~$\WF_\Lambda(u)$ such that
    \[ (x, y, -\xi, -\eta) \in \WF_\Lambda(1 \otimes v). \]
    The inclusion~\eqref{eq:inc-WF} gives~$\xi = 0$ and~$(y, -\eta) \in \WF_\Lambda^0(v)$. Then we conclude two things. On the one hand, since~$(0, \eta) \ne (0, 0)$, we find $\eta \ne 0$ and so $(y, -\eta) \in \WF_\Lambda(v)$. On the other hand, since $(x, y, 0, \eta) \in \WF_\Lambda(u)$, we get~$(x, y, 0, -\eta) \in \WF'_\Lambda(u)$ and so~$(y, -\eta) \in \WF'_\Lambda(u)_Y$. This is a contradiction because, by assumption~\eqref{eq:Gamma}, we have
    \[ (y, -\eta) \in \WF_\Lambda(v) \cap \WF'_\Lambda(u)_Y \subset \Gamma \cap \WF'_\Lambda(u)_Y = \emptyset. \]
    This proves condition~\eqref{eq:cond-produit}. Then by~\cite[Theorem~2.9.8]{CluckersHalupczokLoeserRaibaut} for instance, we can define the product of distributions~$u_v \coloneq (1 \otimes v)u$.

    We consider the kernel $\Ker_v$ of the distribution $u_v$. Let $B$ and $C$ be two polydiscs of~$Y$ such that $\supp v \subset B$ and $\supp v \subset C$. Let us prove the equation
    \begin{equation}\label{eq:indépendance}
        \Ker_v\cf_B = \Ker_v\cf_C.
    \end{equation}
    Since the polydiscs $B$ and $C$ are not disjoint, we can assume that $B \supset C$. Let $\phi$ be a Schwartz--Bruhat function on $X$. We write
    \begin{align*}
        \inner{\Ker_v\cf_B - \Ker_v\cf_C}{\phi} &= \inner{u_v}{\phi \otimes \cf_{B \setminus C}} \\
        &= \inner{\cf_{X \times (B \setminus C)}u_v}{\phi \otimes \cf_{B \setminus C}}.
    \end{align*}
    Denoting by $\delta \colon X \times Y \to (X \times Y) \times (X \times Y)$ the diagonal map, the definition of product of distributions (see~\cite[Theorem~2.9.8]{CluckersHalupczokLoeserRaibaut} for instance) gives
    \[ u_v = \delta^*((1 \otimes v) \otimes u). \]
    where the pullback $\delta^*$ is defined by~\cite[Theorem~2.9.3]{CluckersHalupczokLoeserRaibaut}. By the same theorem, we then obtain
    \[ \supp u_v \subset \delta^{-1}(\supp((1 \otimes v) \otimes u)). \]
    But the definition of tensor product (see~\cite[Theorem~2.7.1]{CluckersHalupczokLoeserRaibaut}) gives immediately
    \[ \supp((1 \otimes v) \otimes u) = (X \times \supp v) \times \supp u \]
    and we get
    \[ \supp u_v \subset (X \times \supp v) \cap \supp u. \]
    Since $[X \times (B \setminus C)] \cap (X \times \supp v) = \emptyset$, we have $[X \times (B \setminus C)] \cap \supp u_v = \emptyset$ and we deduce that
    \[ \inner{\Ker_v\cf_B - \Ker_v\cf_C}{\phi} = 0 \]
    which proves equation~\eqref{eq:indépendance}.

    Let us define the distribution $\Ker v$. By assumption, the support of the distribution $v$ is compact. Let $B$ be a polydisc of $Y$ such that $\supp v \subset B$. With equation~\eqref{eq:indépendance}, the distribution
    \begin{equation}\label{eq:Kv}
        \Ker v \coloneq \Ker_v\cf_B
    \end{equation}
    does not depend on the polydisc $B$. This gives a map
    \[ \Ker \colon \SC'(Y) \cap \Sch'_{\Gamma, \Lambda}(Y) \to \Sch'(X) \]
    which is $\CC$-linear by bilinearity of the product and tensor product of distributions and by linearity of the map associating a distribution to a kernel.

    Let us prove that this map extends the kernel $\Ker \colon \Sch(Y) \to \Sch'(X)$. Let $v$ be a Schwartz--Bruhat function on $Y$. Let $B$ be a polydisc of $Y$ such that $\supp v \subset B$. Let~$\phi$ be a Schwartz--Bruhat function on $X$. Then we can write $\cf_Bv = v$ and we obtain
    \begin{align*}
        \inner{\Ker_v\cf_B}{\phi} &= \inner{u_v}{\phi \otimes \cf_B} \\
        &= \inner{(1 \otimes v)u}{\phi \otimes \cf_B} \\
        &= \inner{u}{\phi \otimes \cf_Bv} \\
        &= \inner{u}{\phi \otimes v} \\
        &= \inner{\Ker v}{\phi}.
    \end{align*}
    This concludes that
    \[ \Ker v = \Ker_v\cf_B \]
    where the left-hand side is given by the map $\Ker \colon \Sch(Y) \to \Sch'(X)$.

    Finally, let $B$ be a polydisc of $Y$. Let us prove that the restriction
    \[ \Ker \colon \SC'_B(Y) \cap \Sch'_{\Gamma, \Lambda}(Y) \to \Sch'(X) \]
    is continuous. Let $(v_j)_{j \geqslant 0}$ be a sequence of $\SC'_B(Y) \cap \Sch'_{\Gamma, \Lambda}(Y)$ and $v$ be a distribution of~$\SC'_B(Y) \cap \Sch'_{\Gamma, \Lambda}(Y)$ such that $v_j \to v$ in $\Sch'_{\Gamma, \Lambda}(Y)$. Let us show that $\Ker v_j \to \Ker v$ in~$\Sch'(X)$. Let $\phi$ be a Schwartz--Bruhat function on $X$. We want to show that
    \begin{equation}\label{eq:convergence}
        \inner{\Ker v_j}{\phi} \to \inner{\Ker v}{\phi}.
    \end{equation}
    For all indexes $j$, with definition~\eqref{eq:Kv}, we write
    \begin{align*}
        \inner{\Ker v_j}{\phi} &= \inner{\Ker_{v_j}\cf_B}{\phi} \\
        &= \inner{u_{v_j}}{\phi \otimes \cf_B}.
    \end{align*}
    But directly from~\cite[Definition~2.7.1]{CluckersHalupczokLoeserRaibaut}, we can show that the $\CC$-linear map
    \[ \fnc{\Sch'(Y)}{\Sch'(X \times Y),}{v}{1 \otimes v} \]
    is continuous and we can deduce that the $\CC$-linear map
    \[ \fnc{\Sch_{\Gamma, \Lambda}'(Y)}{\Sch'(X \times Y),}{v}{u_v = (1 \otimes v)u} \]
    is also continuous since the pullback $\delta^*$ is continuous by~\cite[Theorem~2.9.3]{CluckersHalupczokLoeserRaibaut}. With this and since $v_j \to v$ in $\Sch'(Y)$, we deduce convergence~\eqref{eq:convergence}. This concludes the continuity of the restriction and so the existence part.

    \subproof{Uniqueness}
    Let $\Ker$ and $\Ker'$ be two maps as in Theorem~\ref{thm:extension}. Let $v$ be a distribution of~$\SC'(Y) \cap \Sch'_{\Gamma, \Lambda}(Y)$ and $B$ be a polydisc of $Y$ such that $\supp v \subset B$. By~\cite[Proposition~2.9.2]{CluckersHalupczokLoeserRaibaut}, there exists a sequence $(v_j)_{j \geqslant 0}$ of $\Sch(Y)$ such that $v_j \to v$ in~$\Sch'_{\Gamma, \Lambda}(Y)$. In particular, we have $\cf_Bv_j \to \cf_Bv = v$ in $\Sch'(Y)$. By continuity of the restrictions
    \[ \Ker, \Ker' \colon \SC'_B(Y) \cap \Sch'_{\Gamma, \Lambda}(Y) \to \Sch'(X), \]
    we obtain
    \[ \Ker(\cf_Bv_j) \to \Ker v \qquad \text{and} \qquad
       \Ker'(\cf_Bv_j) \to \Ker'v. \]
    But since the two maps $\Ker$ and $\Ker'$ extends the kernel $\Ker \colon \Sch(Y) \to \Sch'(X)$, we obtain~$\Ker(\cf_Bv_j) = \Ker'(\cf_Bv_j)$ for all indexes $j$. This concludes $\Ker v = \Ker'v$.

    \subproof{Inclusion~(\ref{eq:WF-Kv})}
    Let $v$ be a distribution in $\SC'(Y) \cap \Sch'_{\Gamma, \Lambda}(Y)$. Let $B$ be a polydisc of $Y$ such that $\supp v \subset B$. By Theorem~\ref{thm:WF-p} and with definition~\eqref{eq:Kv}, we get
    \[ \WF_\Lambda(\Ker v)\subset \{(x, \xi) \in X \times (\QQ_p^m \setminus \{0\}) \mid \exists y \in B, (x, y, \xi, 0) \in \WF_\Lambda(u_v)\}. \]
    By~\cite[Theorem~2.9.8]{CluckersHalupczokLoeserRaibaut}, we also get
    \begin{multline*}
        \WF_\Lambda(u_v) \subset \WF_\Lambda^0(u_v) = \{(x, y, \xi' + \xi'', \eta' + \eta'') \in X \times Y \times (\QQ_p^{m + n} \setminus \{0\}) \\
        : (x, y, \xi', \eta') \in \WF_\Lambda^0(1 \otimes v), (x, y, \xi'', \eta'') \in \WF_\Lambda^0(u)\}.
    \end{multline*}
    Let $(x, \xi)$ be a point of $\WF_\Lambda(\Ker v)$. There exists a point $y$ of $B$ such that
    \[ (x, y, \xi, 0) \in \WF_\Lambda(u_v), \]
    and then there exist two elements $\xi'$ and $\xi''$ of $\QQ_p^m$ and an element $\eta'$ of $\QQ_p^n$ such that~$\xi = \xi' + \xi''$ and
    \[ (x, y, \xi', \eta') \in \WF_\Lambda^0(1 \otimes v) \qquad \text{and} \qquad
       (x, y, \xi'', -\eta') \in \WF_\Lambda^0(u). \]
    Since $(x, y, \xi', \eta') \in \WF_\Lambda^0(1 \otimes v)$, we distinguish two cases.
    \begin{itemize}
        \item We assume that $(x, y, \xi', \eta') \in \WF_\Lambda^0(1 \otimes v)$. Then by inclusion~\eqref{eq:inc-WF}, we deduce that~$\xi' = 0$. We obtain $\xi = \xi'' \ne 0$ which leads to $(x, y, \xi, -\eta') \in \WF_\Lambda(u)$ and
        \[ (x, y, \xi, \eta') \in \WF'_\Lambda(u). \]
        Moreover, still by inclusion~\eqref{eq:inc-WF}, we have $(y, \eta') \in \WF_\Lambda^0(v)$.
        \begin{itemize}
            \item If $\eta' = 0$, then $(x, y, \xi, 0) \in \WF_\Lambda(u)$, so $(x, \xi) \in \WF_\Lambda(u)_X$.

            \item If $\eta' \ne 0$, then $(y, \eta') \in \WF_\Lambda(v)$, so $(x, \xi) \in \WF'_\Lambda(u) \circ \WF_\Lambda(v)$.
        \end{itemize}

        \item We assume that $(x, y, \xi', \eta') \in X \times Y \times \{0\}$. Then $\xi' = 0$ and $\eta' = 0$ which leads to~$\xi = \xi'' \ne 0$. We deduce that $(x, y, \xi, 0) \in \WF_\Lambda(u)$, so $(x, \xi) \in \WF_\Lambda(u)_X$.
    \end{itemize}
    This concludes inclusion~\eqref{eq:WF-Kv}.
\end{proof}

\section{Kernels of motivic distributions}
\label{sec:mot}

\subsection{Motivic integration}

We will work with Cluckers--Loeser's motivic integration~\cite{CluckersLoeser2008,CluckersLoeser2010}. We recall briefly the objects introduced in these articles and we refer to those for more details.

\subsubsection{A brief overview}

We fix a characteristic zero field $k$. We consider a Denef--Pas language $\L_\text{DP}$, i.e. a language with three sorts: the valued field sort (equipped with the ring language), the residual field sort (equipped with an extension of the ring language by some constant symbols) and the value group sort (equipped with an extension of the Presburger language by some constant symbols) endowed with a symbol for the valuation $\ord$ and one for the angular map $\ac$. We consider the extension~$\L_\text{DP}(k)$ of this language by adding constant symbols for each element of the sets $k\ls t$ and $k$, respectively in the valued field sort and in the residue field sort.

Let $\Field_k$ be the category of fields which contain the field $k$. Let $n$, $\ell$ and $r$ be three nonnegative integers. For a field $K$ of $\Field_k$, we set
\[ h[n, \ell, r](K) \coloneqq K\ls t^n \times K^\ell \times \ZZ^r. \]
A \emph{definable subassignment} $X$, also called \emph{definable set}, of $h[n, \ell, r]$ is a collection of subsets~$X(K)$ of $h[n, \ell, r](K)$ for each field $K$ of $\Field_k$ such that there exists an~$\L_\text{DP}(k)$-formula $\phi(x, \xi, \alpha)$ such that, for any field $K$ of $\Field_k$, we have
\[ X(K) = \{(x, \xi, \alpha) \in h[n, \ell, r](K) \mid (K\ls t, K, \ZZ) \models \phi(x, \xi, a)\}. \]
A \emph{definable morphism} between two definable subassignments $X$ and $Y$ is a collection of maps $f_K \colon X(K) \to Y(K)$ for each field $K$ of $\Field_k$ such that their graphs form a definable subassignment. A \emph{point} of a definable subassignment $X$ is a data of a field~$K$ of $\Field_k$ and an element~$x_0$ of $X(K)$. In this case, we will write $(x_0, K) \in X$%
    \footnote{The correct notation would be $(x_0, K) \in \points X$, but for simplicity, we will omit the bars.}
and we denote by $\points X$ the set of points of $X$.

Let $\Def_k$ be the category of definable subassignments over the field $k$. We denote by $\{*\}$ its final object. For a definable subassignment $X$, we denote by
\begin{itemize}
    \item $X[n, \ell, r]$ the product $X \times h[n, \ell, r]$ for some nonnegative integers $n$, $\ell$ and $r$;

    \item $\Def_X$ the category of objects of $\Def_k$ over $X$, i.e. of definable morphisms~$Y \to X$ for some definable subassignments $Y$;

    \item $\Cons(X)\expe$ the ring of exponential constructible motivic functions on $X$;
\end{itemize}
We denote by $\LL$ the Lefschetz symbol and by $\Exp(-)$ the exponential at the level of valued field.

Let $P$ be a definable subassignment. For an object $p \colon X \to P$ of~$\Def_P$ whose all fibers have the same dimension $d$, we denote by~$\Inte_P(X)\expe$ or $\Inte_p(X)\expe$ the $\Cons(P)\expe$-submodule of $\Cons(X)\expe$ of $p$-integrable exponential constructible motivic functions on~$X$ with parameter relative to $P$ (see~\cite[§7]{CluckersLoeser2010}).

For a definable morphism $f \colon X \to Y$, we denote by~$f^* \colon \Cons(Y)\expe \to \Cons(X)\expe$ its pullback morphism. For an object $f \colon X \to P$ of $\Def_P$, we denote by
\[ f_! \colon \Inte_P(X)\expe \to \Cons(P)\expe \]
its pushforward morphism (see~\cite[Theorem~14.1.1]{CluckersLoeser2008} and~\cite[Theorem~4.3.1 \&{}~§7]{CluckersLoeser2010}).

For a definable morphism $\alpha \colon P \to h[0, 0, 1]$, we denote by $\Ball_P(0, \alpha)$ the \emph{ball of radius $\alpha$ centered at the origin}, that is
\[ \Ball_P(0, \alpha) \coloneq \{(p, x) \in P[1, 0, 0] \mid \ord x \geqslant \alpha(p)\}. \]
When $P = \{*\}$, this ball will be simply denoted by $\Ball(0, \alpha)$. For a positive integer~$n$, we denote by $\Ball_P(0, \alpha)^n$ the fiber product of $n$ copies of~$\Ball_P(0, \alpha)$ over $P$. For a definable morphism~$c \colon P \to h[n, 0, 0]$, we set
\[ \Ball_P(c, \alpha) \coloneq \{(p, x) \in P[n, 0, 0] \mid \ord(x - c(p)) \geqslant \alpha(p)\} \]
where $\ord(y_1, \dots, y_n) \coloneq \min(\ord y_1, \dots, \ord y_n)$ for a point $(y_1, \dots, y_n)$ of~$h[n, 0, 0]$. With the canonical projection to $P$, these are objects of $\Def_P$.

Also in this section, we will use the notion of evaluation of constructible motivic functions introduced by Cluckers--Halupczok~\cite{CluckersHalupczok}. We should think evaluation in the natural way. In particular, we can recover constructible motivic functions from their evaluations (see~\cite[Theorem~1]{CluckersHalupczok})

\subsubsection{Convolution product}

To begin, let us recall the definition of the convolution product for constructible functions introduced by Cluckers--Loeser~\cite{CluckersLoeser2010}. Let $P$ be a definable subassignment and $n$ be a positive integer. We consider the two canonical projections~$\pi_1, \pi_2 \colon P[2n, 0, 0] \to P[n, 0, 0]$ and the sum morphism
\[ \fonc{s}{P[2n, 0, 0]}{P[n, 0, 0],}{(p, x, y)}{(p, x + y).} \]

\begin{defi}[Cluckers--Loeser~{\cite[Proposition-Definition~7.4.1]{CluckersLoeser2010}}]\label{def:convol}
    Let $f$ and $g$ be two functions of~$\Inte_P(P[n, 0, 0])\expe$. Then the function $\pi_1^*f \cdot \pi_2^*g$ of~$\Cons(P[2n, 0, 0])\expe$ lies in $\Inte_s(P[2n, 0, 0])\expe$ and the function
    \[ f * g \coloneq s_!(\pi_1^*f \cdot \pi_2^*g) \]
    of $\Cons(P[n, 0, 0])\expe$ lies in $\Inte_P(P[n, 0, 0])\expe$.
\end{defi}

The function $f * g$ is called the \emph{convolution product} of the functions $f$ and $g$. We give another expression of the convolution product (see~\cite[Remark~2.19]{Raibaut}).

\begin{prop}\label{prop:convol}
    Let $f$ and $g$ be two functions of~$\Inte_P(P[n, 0, 0])\expe$. We consider the definable morphism
    \[ \fonc{d}{P[2n, 0, 0]}{P[n, 0, 0],}{(p, z, y)}{(p, y - z).} \]
    Then
    \[ f * g = \pi_{2!}(\pi_1^*f \cdot d^*g). \]
\end{prop}

\begin{proof}
    Definition~\ref{def:convol} gives
    \[ f * g = s_!(\pi_1^*f \cdot \pi_2^*g). \]
    We consider the definable isomorphism
    \[ \fonc{h}{P[2n, 0, 0]}{P[2n, 0, 0],}{(p, z, y)}{(p, z, y - z).} \]
    whose order of Jacobian is zero. By change of variables theorem (see~\cite[Theorem~4.2.1]{CluckersLoeser2010}), we obtain
    \[ \pi_1^*f \cdot \pi_2^*g = h_!(h^*(\pi_1^*f \cdot \pi_2^*g)). \]
    As $s \circ h = \pi_2$, $\pi_1 \circ h = \pi_1$ and $\pi_2 \circ h = d$, by Fubini's theorem (see~\cite[Remark~4.3.2]{CluckersLoeser2010}), we deduce that
    \begin{align*}
        f * g &= (s \circ h)_!(h^*(\pi_1^*f \cdot \pi_2^*g)) \\
        &= \pi_{2!}(h^*(\pi_1^*f \cdot \pi_2^*g)) \\
        &= \pi_{2!}((\pi_1 \circ h)^*f \cdot (\pi_2 \circ h)^*g) \\
        &= \pi_{2!}(\pi_1^*f \cdot d^*g)
    \end{align*}
    which gives the desired formula.
\end{proof}

\begin{rema}
    By adopting the integral notation and evaluating functions, the formula given by Proposition~\ref{prop:convol} can be rewritten as
    \[ f * g(p, y) = \int_{h[n, 0, 0]} f(p, z)g(p, y - z) \dd z \]
    for all points $(p, y)$ of $P[n, 0, 0]$. Then this convolution product is the analogous version of the one in real analysis.
\end{rema}

\subsubsection{Schwartz--Bruhat functions}

Let $n$ be a positive integer. We recall the notion of Schwartz--Bruhat functions on $P[n, 0, 0]$ with parameter relative to $P$.

\begin{defi}[Cluckers--Loeser~{\cite[§7.5]{CluckersLoeser2010}}]\label{def:SB}
    A \emph{Schwartz--Bruhat function} on $P[n, 0, 0]$ with parameter relative to $P$ is a function $\varphi$ of $\Inte_P(P[n, 0, 0])\expe$ which satisfies the two following points:
    \begin{itemize}
        \item there exists a definable morphism $\alpha^- \colon P \to h[0, 0, 1]$ such that, for any definable morphism $\alpha \colon P \to h[0, 0, 1]$ such that $\alpha \leqslant \alpha^-$, we have
        \begin{equation}\label{eq:bounded-support}
            \varphi \cdot \cf_{\Ball_P(0, \alpha)^n} = \varphi;
        \end{equation}

        \item there exists a definable morphism $\alpha^+ \colon P \to h[0, 0, 1]$ such that, for any definable morphism $\alpha \colon P \to h[0, 0, 1]$ such that $\alpha \geqslant \alpha^+$, we have
        \begin{equation}\label{eq:locally-constant}
            \varphi * \cf_{\Ball_P(0, \alpha)^n} = \LL^{-\alpha n}\varphi.
        \end{equation}
    \end{itemize}
\end{defi}

For a Schwartz--Bruhat function $\varphi$ on $P[n, 0, 0]$, we denote by $\alpha^-(\varphi)$ (resp. $\alpha^+(\varphi)$) a definable morphism $\alpha^-$ (resp. $\alpha^+$) as in Definition~\ref{def:SB}. It is just a notation to avoid introducing these morphisms ---~such definable morphisms $\alpha^-$ and $\alpha^+$ are not unique. We denote by $\Sch_P(P[n, 0, 0])\expe$ the $\Cons(P)\expe$-module of Schwartz--Bruhat functions on~$P[n, 0, 0]$ with parameter relative to $P$.

\begin{rema}
    In Definition~\ref{def:SB}, informally, the first point ensures that the support of the function $\varphi$ is bounded and the second point ensures that the function $\varphi$ is locally constant. This corresponds closely to the $p$-adic case described above.
\end{rema}

\subsubsection{Tensor product on Schwartz--Bruhat functions}

We define a tensor product on Schwartz--Bruhat functions. Let $n_1$ and $n_2$ be two positive integers. We will use the following definable isomorphism
\[ P[n_1 + n_2, 0, 0] \cong P[n_1, 0, 0] \times P[n_2, 0, 0]. \]
Let $p_1 \colon P[n_1 + n_2, 0, 0] \to P[n_1, 0, 0]$ and $p_2 \colon P[n_1 + n_2, 0, 0] \to P[n_2, 0, 0]$ be the two canonical projections.

\begin{defi}\label{def:tensor-product}
    Let $\varphi_1$ be a function of $\Sch_P(P[n_1, 0, 0])\expe$ and $\varphi_2$ be a function of~$\Sch_P(P[n_2, 0, 0])\expe$. Their \emph{tensor product} is the function
    \begin{equation}\label{eq:tensor-product}
        \varphi_1 \otimes \varphi_2 \coloneq p_1^*\varphi_1 \cdot p_2^*\varphi_2
    \end{equation}
    of $\Inte_P(P[n_1 + n_2, 0, 0])\expe$.
\end{defi}

\begin{rema}
    The fact that the function $\varphi_1 \otimes \varphi_2$ lies in $\Inte_P(P[n_1 + n_2, 0, 0])\expe$ comes from the Fubini's theorem. Evaluating functions, this function is given by the equation
    \[ \varphi_1 \otimes \varphi_2(p, x, y) = \varphi_1(p, x)\varphi_2(p, y) \]
    for all points $(p, x, y)$ of $P[n_1 + n_2, 0, 0]$. This is the analogous notion of the tensor product of $p$-adic Schwartz--Bruhat functions as in~§\ref{sec:p}.
\end{rema}

\begin{prop}\label{prop:tenseur-SB}
    With notation of Definition~\ref{def:tensor-product}, the function $\varphi_1 \otimes \varphi_2$ is a Schwartz--Bruhat function on $P[n_1 + n_2, 0, 0]$.
\end{prop}

\begin{proof}
    Let us show that the support of this function is bounded, namely there exists a definable morphism $\alpha \colon P \to h[0, 0, 1]$ such that
    \[ \varphi_1 \otimes \varphi_2 \cdot \cf_{\Ball_P(0, \alpha)^{n_1 + n_2}} = \varphi_1 \otimes \varphi_2. \]
    Let $\alpha \colon P \to h[0, 0, 1]$ be a definable morphism such that
    \[ \alpha \leqslant \alpha^-(\varphi_1) \qquad \text{and} \qquad
       \alpha \leqslant \alpha^-(\varphi_2). \]
    Thus, we can write these two equalities
    \[ \varphi_1 \cdot \cf_{\Ball_P(0, \alpha)^{n_1}} = \varphi_1 \qquad \text{and} \qquad
       \varphi_2 \cdot \cf_{\Ball_P(0, \alpha)^{n_2}} = \varphi_2. \]
    Since the maps $p_1^*$ and $p_2^*$ are ring morphisms, we get
    \begin{align*}
        \varphi_1 \otimes \varphi_2 &= p_1^*(\varphi_1 \cdot \cf_{\Ball_P(0, \alpha)^{n_1}})\cdot p_2^*(\varphi_2 \cdot \cf_{\Ball_P(0, \alpha)^{n_2}}) \\
        &= p_1^*\varphi_1 \cdot p_2^*\varphi_2 \cdot p_1^*\cf_{\Ball_P(0, \alpha)^{n_1}} \cdot p_2^*\cf_{\Ball_P(0, \alpha)^{n_2}} \\
        &= \varphi_1 \otimes \varphi_2 \cdot (\cf_{\Ball_P(0, \alpha)^{n_1}} \circ p_1) \cdot (\cf_{\Ball_P(0, \alpha)^{n_2}} \circ p_2) \\
        &= \varphi_1 \otimes \varphi_2 \cdot \cf_{\Ball_P(0, \alpha)^{n_1} \times_P P[n_2, 0, 0]}\cf_{P[n_1, 0, 0] \times_P \Ball_P(0, \alpha)^{n_2}} \\
        &= \varphi_1 \otimes \varphi_2 \cdot \cf_{\Ball_P(0, \alpha)^{n_1 + n_2}}.
    \end{align*}

    Let us show that this function is locally constant, that is there exists a definable morphism $\alpha \colon P \to h[0, 0, 1]$ such that
    \[ \varphi_1 \otimes \varphi_2 * \cf_{\Ball_P(0, \alpha)^{n_1 + n_2}} = \LL^{-\alpha(n_1 + n_2)}\varphi_1 \otimes \varphi_2. \]
    Let $\alpha \colon P \to h[0, 0, 1]$ be a definable morphism such that
    \[ \alpha \geqslant \alpha^+(\varphi_1) \qquad \text{and} \qquad
       \alpha \geqslant \alpha^+(\varphi_2). \]
    Thus, by adopting integral notation, we can write these two equalities
    \begin{align*}
        \varphi_1 * \cf_{\Ball_P(0, \alpha)^{n_1}}(x) = \int_{P[n_1, 0, 0]} \varphi_1(z)\cf_{\Ball_P(0, \alpha)^{n_1}}(x - z) \dd z &= \LL^{-\alpha n_1}\varphi_1(x) \qquad \text{and} \qquad \\
        \varphi_2 * \cf_{\Ball_P(0, \alpha)^{n_2}}(y) = \int_{P[n_2, 0, 0]} \varphi_2(t)\cf_{\Ball_P(0, \alpha)^{n_2}}(y - t) \dd t &= \LL^{-\alpha n_2}\varphi_2(y).
    \end{align*}
    Thanks to axioms of the motivic integral and Fubini's theorem, we get
    \begin{align*}
        &(\varphi_1 \otimes \varphi_2) * \cf_{\Ball_P(0, \alpha)^{n_1 + n_2}}(x, y) \\
        &\qquad= \int_{P[n_1 + n_2, 0, 0]} \varphi_1 \otimes \varphi_2(z, t)\cf_{\Ball_P(0, \alpha)^{n_1 + n_2}}(x - z, y - t) \dd z \dd t \\
        &\qquad= \int_{P[n_1 + n_2, 0, 0]} \varphi_1(z)\varphi_2(z)\cf_{\Ball_P(0, \alpha)^{n_1}}(x - z)\cf_{\Ball_P(0, \alpha)^{n_2}}(y - t) \dd z \dd t \\
        &\qquad= \int_{P[n_1, 0, 0]} \varphi_1(z)\cf_{\Ball_P(0, \alpha)^{n_1}}(x - z) \dd z \\
        &\qquad\qquad\qquad \cdot \int_{P[n_2, 0, 0]} \varphi_2(t)\cf_{\Ball_P(0, \alpha)^{n_2}}(y - t) \dd t \\
        &\qquad= \LL^{-\alpha n_1}\varphi_1(x)\LL^{-\alpha n_2}\varphi_2(y) \\
        &\qquad= \LL^{-\alpha(n_1 + n_2)}[p_1^*\varphi_1 \cdot p_2^*\varphi_2](x, y) \\
        &\qquad= \LL^{-\alpha(n_1 + n_2)}\varphi_1 \otimes \varphi_2(x, y).
    \end{align*}
    This concludes that the function $\varphi_1 \otimes \varphi_2$ is a Schwartz--Bruhat function.
\end{proof}

\begin{rema}
    This last proof also tells that we can take
    \begin{align*}
        \alpha^-(\varphi_1 \otimes \varphi_2) &= \min(\alpha^-(\varphi_1), \alpha^-(\varphi_2)) \qquad \text{and} \qquad \\
       \alpha^+(\varphi_1 \otimes \varphi_2) &= \max(\alpha^+(\varphi_1), \alpha^+(\varphi_2)).
    \end{align*}
\end{rema}

\begin{rema}
    With Proposition~\ref{prop:tenseur-SB}, we get a $\Cons(P)\expe$-bilinear map
    \[ - \otimes - \colon \Sch_P(P[n_1, 0, 0])\expe \times \Sch_P(P[n_2, 0, 0])\expe \to \Sch_P(P[n_1 + n_2, 0, 0])\expe. \]
    Contrary to the $p$-adic case, the induced $\Cons(P)\expe$-linear map
    \[ - \otimes - \colon \Sch_P(P[n_1, 0, 0])\expe \otimes_{\Cons(P)\expe} \Sch_P(P[n_2, 0, 0])\expe \to \Sch_P(P[n_1 + n_2, 0, 0])\expe. \]
    does not appear to be surjective. Consequently, to find a motivic Schwartz kernel theorem, we will need to take a different approach.
\end{rema}

\subsubsection{Definable distributions}

We recall the definition of motivic distributions introduced by Raibaut~\cite{Raibaut}. Let $n$ be a positive integer.

Let $W$ and $W'$ be two definable subassignments and~$g \colon W \to W'$ be a definable morphism. We set the definable morphism~$g \times \id_{h[n, 0, 0]} \colon W[n, 0, 0] \to W'[n, 0, 0]$ defined by the equality
\[ (g \times \id_{h[n, 0, 0]})(w, x) = (g(w), x) \]
for all points $(w, x)$ of $W[n, 0, 0]$.

\begin{defi}[Raibaut~{\cite[Definition~5.4]{Raibaut}}]
    A function $\varphi$ of $\Sch_W(W[n, 0, 0])\expe$ is~\emph{$g \times \id_{h[n, 0, 0]}$-convenient} if
    \begin{enumerate}
        \item it is $g \times \id_{h[n, 0, 0]}$-integrable and $(g \times \id_{h[n, 0, 0]})_!\varphi \in \Inte_{W'}(W'[n, 0, 0])\expe$;

        \item there exists a definable morphism $\beta^- \colon W' \to h[0, 0, 1]$ such that
        \[ \varphi \cdot \cf_{\Ball_W(0, \beta^- \circ g)^n} = \varphi; \]

        \item there exists a definable morphism $\beta^+ \colon W' \to h[0, 0, 1]$ such that
        \[ \varphi * \cf_{\Ball_W(0, \beta^+ \circ g)^n} = \LL^{-(\beta^+ \circ g)n}\varphi. \]
    \end{enumerate}
\end{defi}

\begin{rema}
    For any $g \times \id_{h[n_1, 0, 0]}$-convenient function $\varphi_1$ of $\Sch_W(W[n_1, 0, 0])\expe$ and any $g \times \id_{h[n_2, 0, 0]}$-convenient function $\varphi_2$ of $\Sch_W(W[n_2, 0, 0])\expe$, the function~$\varphi_1 \otimes \varphi_2$ is not necessary $g \times \id_{h[n, 0, 0]}$-convenient.
\end{rema}

Let $P$ be a definable subassignment.

\begin{defi}[Raibaut~{\cite[Definition~5.7]{Raibaut}}]\label{def:definable-distribution}
    A \emph{definable distribution} on $P[n, 0, 0]$ with parameter relative to $P$ is a collection $(u_{\Phi_W})_{\Phi_W \in \Def_P}$ such that, for each object~$\Phi_W \colon W \to P$ of $\Def_P$, the element $u_{\Phi_W}$ is a~$\Cons(W)\expe$-linear map
    \[ \fonc{u_{\Phi_W}}{\Sch_W(W[n, 0, 0])\expe}{\Cons(W)\expe,}{\varphi}{\inner{u_{\Phi_W}}{\varphi}} \]
    such that, for any objects $\Phi_W \colon W \to P$ and $\Phi_{W'} \colon W' \to P$ of $\Def_P$ and any morphism~$g \colon W \to W'$ of $\Def_P$, the following points are satisfied:
    \begin{enumerate}[label=(D\arabic*)]
        \item\label{it:pullback} for each function $\varphi$ of $\Sch_{W'}(W'[n, 0, 0])\expe$, we have
        \[ g^*\inner{u_{\Phi_{W'}}}{\varphi} = \inner{u_{\Phi_W}}{(g \times \id_{h[n, 0, 0]})^*\varphi}; \]

        \item\label{it:pushforward} for each $g \times \id_{h[n, 0, 0]}$-convenient function $\varphi$ of $\Sch_W(W[n, 0, 0])\expe$, the function~$\inner{u_{\Phi_W}}{\varphi}$ is $g$-integrable and
        \[ g_!\inner{u_{\Phi_W}}{\varphi} = \inner{u_{\Phi_{W'}}}{(g \times \id_{h[n, 0, 0]})_!\varphi}. \]
    \end{enumerate}
\end{defi}

We denote by $\Sch_P'(P[n, 0, 0])\expe$ the set of definable distributions on $P[n, 0, 0]$ with parameter relative to $P$. With~\cite[Remark~5.11]{Raibaut}, it is endowed with a structure of~$\Cons(P)\expe$-module. When $P = \{*\}$, this set will be simply denoted by $\Sch'(h[n, 0, 0])\expe$.

\begin{rema}
    In Definition~\ref{def:definable-distribution}, when we say that the morphism $g \colon W \to W'$ is a morphism of $\Def_P$, we implicitly impose that it commutes with the structural morphisms $\Phi_W$ and $\Phi_{W'}$, namely the diagram
    \begin{center}
        \begin{tikzcd}
            & W' \arrow[d, "\Phi_{W'}"] \\
            W \arrow[ur, "g"] \arrow[r, "\Phi_W"'] & P
        \end{tikzcd}
    \end{center}
    commutes. In fact, we should rather write $g \colon \Phi_W \to \Phi_{W'}$.

    Moreover, in items~\ref{it:pullback} and~\ref{it:pushforward}, the functions~$(g \times \id_{h[n, 0, 0]})^*\varphi$ and~$(g \times \id_{h[n, 0, 0]})_!\varphi$ are Schwartz--Bruhat functions by~\cite[Lemmas~5.2 \&~5.5]{Raibaut}. These two items will be respectively called the \emph{pullback condition} and the \emph{pushforward condition}.
\end{rema}

\subsection{Motivic Schwartz kernel theorem}
\label{ssec:Schwartz}

We can now state a Schwartz kernel theorem for definable distributions. Let $n_1$ and $n_2$ be two positive integers and $P$ be a definable subassignment. We set $n \coloneq n_1 + n_2$. We first introduce the notion of kernel.

\paragraph*{Notation} Let $W$ be a definable subassignment. We denote by $\Sch_W^*(W[n_1, 0, 0])\expe$ the dual of the~$\Cons(W)\expe$-module $\Sch_W(W[n_1, 0, 0])\expe$. For $i = 1, 2$, we consider the definable morphism
\[ \fonc{d_i}{W[2n_i, 0, 0]}{W[n_i, 0, 0],}{(w, z, y)}{(w, y - z)} \]
and the canonical projections
\begin{align*}
    q_i \colon W[n + n_i, 0, 0] = W[n_1 + n_2 + n_i, 0, 0] &\to W[2n_i, 0, 0] \qquad \text{and} \qquad \\
    p \colon W[n, 0, 0] &\to W.
\end{align*}

We will adopt similar notation for a definable subassignment $W'$: we add a prime on these morphisms (like $d_i'$, $q_i'$ or $p'$).

\begin{defi}
    A \emph{kernel} on $h[n_1, 0, 0] \times h[n_2, 0, 0]$ with parameter relative to $P$ is a collection $(\Ker_{\Phi_W})_{\Phi_W \in \Def_P}$ such that, for each object~$\Phi_W \colon W \to P$ of $\Def_P$, the element $\Ker_{\Phi_W}$ is a $\Cons(W)\expe$-linear map
    \[ \Ker_{\Phi_W} \colon \Sch_W(W[n_2, 0, 0])\expe \to \Sch_W^*(W[n_1, 0, 0])\expe \]
    such that, for each objects $\Phi_W \colon W \to P$ and $\Phi_{W'} \colon W' \to P$ of $\Def_P$ and each morphism~$g \colon W \to W'$ of $\Def_P$, the following points are satisfied.
    \begin{enumerate}[label=(K\arabic*)]
        \item\label{it:integrable} Let $\varphi$ be a function of $\Sch_W(W[n, 0, 0])\expe$ and~$\alpha \colon W \to h[0, 0, 1]$ be a definable morphism such that
        \[ \varphi * \cf_{\Ball_W(0, \alpha)^n} = \LL^{-\alpha n}\varphi. \]
        Then
        \begin{itemize}
            \item the function
            \[ F_{\Ker, \varphi, \alpha} \coloneq \LL^{(\alpha \circ p)n}\varphi\inner{\Ker_{\Phi_W \circ p}(d_2 \circ q_2)^*\cf_{\Ball_W(0, \alpha)^{n_2}}}{(d_1 \circ q_1)^*\cf_{\Ball_W(0, \alpha)^{n_1}}} \]
            is $p$-integrable%
                \footnote{We refer to Remark~\ref{rem:informal} for an informal interpretation of the function $F_{\Ker, \varphi, \alpha}$.};

            \item if the function $\varphi$ is $g \times \id_{h[n, 0, 0]}$-convenient, then the function $F_{\Ker, \varphi, \alpha}$ is also $g \circ p$-integrable.
        \end{itemize}

        \item\label{it:pullback-K} Let $\varphi_1$ and $\varphi_2$ be functions of $\Sch_{W'}(W'[n_1, 0, 0])\expe$ and~$\Sch_{W'}(W'[n_2, 0, 0])\expe$ respectively. Then
        \begin{equation}\label{eq:pullback-K}
            g^*\inner{\Ker_{\Phi_{W'}}\varphi_2}{\varphi_1} = \inner{\Ker_{\Phi_W}(g \times \id_{h[n_2, 0, 0]})^*\varphi_2}{(g \times \id_{h[n_1, 0, 0]})^*\varphi_1}.
        \end{equation}

        \item\label{it:pushforward-K} Let $\varphi_1$ and $\varphi_2$ be functions of $\Sch_W(W[n_1, 0, 0])\expe$ and~$\Sch_W(W[n_2, 0, 0])\expe$ respectively. If the function $\varphi_1 \otimes \varphi_2$ is~$g \times \id_{h[n, 0, 0]}$-convenient, then the function~$\inner{\Ker_{\Phi_W}\varphi_2}{\varphi_1}$ is $g$-integrable and
        \begin{equation}\label{eq:pushforward-K}
            g_!\inner{\Ker_{\Phi_W}\varphi_2}{\varphi_1} = p'_!F_{\Ker, \psi, \alpha}
        \end{equation}
        where we consider
        \begin{itemize}
            \item the Schwartz--Bruhat function
            \[ \psi \coloneq (g \times \id_{h[n, 0, 0]})_!(\varphi_1 \otimes \varphi_2), \]

            \item a definable morphism $\alpha \colon W' \to h[0, 0, 1]$ such that
            \[ \psi * \cf_{\Ball_{W'}(0, \alpha)^n} = \LL^{-\alpha n}\psi, \]

            \item the canonical projection $p' \colon W'[n, 0, 0] \to W'$ and

            \item the function
            \[ F_{\Ker, \psi, \alpha} \coloneq \LL^{(\alpha \circ p')n}\psi\inner{\Ker_{\Phi_{W'} \circ p'}(d_2' \circ q_2')^*\cf_{\Ball_{W'}(0, \alpha)^{n_2}}}{(d_1' \circ q_1')^*\cf_{\Ball_{W'}(0, \alpha)^{n_1}}}. \]
        \end{itemize}

        \item\label{it:independance-alpha-plus} Let $\varphi$ be a function of $\Sch_W(W[n, 0, 0])\expe$ and $\alpha, \beta \colon W \to h[0, 0, 1]$ be two definable morphisms such that
        \begin{align*}
            \varphi * \cf_{\Ball_W(0, \alpha)^n} &= \LL^{-\alpha n}\varphi \qquad \text{and} \qquad \\
            \varphi * \cf_{\Ball_W(0, \beta)^n} &= \LL^{-\beta n}\varphi.
        \end{align*}
        Then
        \begin{equation}\label{eq:independance-alpha-plus}
            p_!F_{\Ker, \varphi, \alpha} = p_!F_{\Ker, \varphi, \beta}
        \end{equation}
        where we consider the canonical projection $p \colon W[n, 0, 0] \to W$.
    \end{enumerate}
\end{defi}

\begin{rema}
    Hypothesis~\ref{it:integrable}--\ref{it:independance-alpha-plus} will be discussed in Remark~\ref{rem:discussion-hypothesis}.
\end{rema}

We can now state a motivic version of the $p$-adic Schwartz kernel theorem (see Theorem~\ref{thm:kernel-p}). As a first step, we prove the direct sense of such a result. Once again, we consider a definable subassignment $P$ and two positive integers $n_1$ and $n_2$. Also we set $n \coloneq n_1 + n_2$.

\begin{theo}\label{thm:kernel-mot}
    Let $u$ be a distribution of $\Sch_P'(P[n, 0, 0])\expe$. Then there exists a unique kernel $\Ker = (\Ker_{\Phi_W})_{\Phi_W \in \Def_P}$ on $h[n_1, 0, 0] \times h[n_2, 0, 0]$ such that, for any object~$\Phi_W \colon W \to P$ of $\Def_P$ and for any functions $\varphi_1$ of $\Sch_W(W[n_1, 0, 0])\expe$ and $\varphi_2$ of $\Sch_W(W[n_2, 0, 0])\expe$, the equality
    \begin{equation}\label{eq:noyau-mot}
        \inner{u_{\Phi_W}}{\varphi_1 \otimes \varphi_2} = \inner{\Ker_{\Phi_W}\varphi_2}{\varphi_1}
    \end{equation}
    holds. Moreover, for any object $\Phi_W \colon W \to P$ of $\Def_P$, for any function $\varphi$ of~$\Sch_W(W[n, 0, 0])\expe$ and for any definable morphism~$\alpha \colon W \to h[0, 0, 1]$ such that
    \[ \varphi * \cf_{\Ball_W(0, \alpha)^n} = \LL^{-\alpha n}\varphi, \]
    we have
    \begin{equation}\label{eq:rel-u-K}
        \inner{u_{\Phi_W}}{\varphi} = p_!F_{\Ker, \varphi, \alpha}
    \end{equation}
    where we consider the canonical projection $p \colon W[n, 0, 0] \to W$.
\end{theo}

\begin{defi}
    With notation of Theorem~\ref{thm:kernel-mot}, the family $\Ker \coloneq (\Ker_{\Phi_W})_{\Phi_W \in \Def_P}$ is the \emph{kernel} of the distribution $u$.
\end{defi}

\begin{rema}\label{rem:informal}
    Before proving this theorem, we give some informal interpretations of the objects involved in equation~\eqref{eq:rel-u-K} and condition~\ref{it:pullback-K} using the integral notation. The function~$F_{\Ker, \varphi, \alpha}$ can be interpreted as
    \[ (w, x_1, x_2) \longmapsto \LL^{\alpha(w)n}\varphi(w, x_1, x_2)\inner{\Ker_{\Phi_W \circ p}\cf_{\Ball(0, \alpha(w))^{n_2}}(x_2 - \cdot)}{\cf_{\Ball(0, \alpha(w))^{n_1}}(x_1 - \cdot)} \]
    and equation~\eqref{eq:rel-u-K} can be rewritten as
    \begin{multline*}
        \inner{u_{\Phi_W}}{\varphi}(w) = \LL^{\alpha(w)n} \int_{h[n_1, 0, 0]} \int_{h[n_2, 0, 0]} \varphi(w, x_1, x_2) \\
        \inner{\Ker_{\Phi_W \circ p}\cf_{\Ball(0, \alpha(w))^{n_2}}(x_2 - \cdot)}{\cf_{\Ball(0, \alpha(w))^{n_1}}(x_1 - \cdot)} \dd x_1 \dd x_2.
    \end{multline*}
    This equation is like an average formula, similarly to~\cite[Proposition~5.17]{Raibaut}. Moreover, equation~\eqref{eq:pullback-K} can be interpreted as
    \[ \inner{\Ker_{\Phi_{W'}}\varphi_2}{\varphi_1}(g(w)) = \inner{\Ker_{\Phi_W}[x_2 \longmapsto \varphi_2(g(w), x_2)]}{x_1 \longmapsto \varphi_1(g(w), x_1)}. \]
\end{rema}

\begin{proof}[Proof of Theorem~\ref{thm:kernel-mot}]
    Let $\Phi_W \colon W \to P$ be an object of $\Def_P$. We define the map~$\Ker_{\Phi_W}$ by equality~\eqref{eq:noyau-mot}. Thanks to the $\Cons(W)\expe$-linearity of the map $u_{\Phi_W}$ and since pullbacks are ring morphisms, this latter map is also $\Cons(W)\expe$-linear and its image is included in the~$\Cons(W)\expe$-module $\Sch_W^*(W[n_1, 0, 0])\expe$. This proves equation~\eqref{eq:noyau-mot}. Moreover, the uniqueness comes from equality~\eqref{eq:noyau-mot}.

    \sepproof

    We start by proving equation~\eqref{eq:rel-u-K}. Let $\Phi_W \colon W \to P$ be an object of $\Def_P$. Let~$\varphi$ be a function of $\Sch_W(W[n, 0, 0])\expe$ and~$\alpha \colon W \to h[0, 0, 1]$ be a definable morphism such that
    \begin{equation}\label{eq:convol-phi}
        \varphi * \cf_{\Ball_W(0, \alpha)^n} = \LL^{-\alpha n}\varphi.
    \end{equation}
    Let us prove that the function $F_{\Ker, \varphi, \alpha}$ is $p$-integrable and
    \[ \inner{u_{\Phi_W}}{\varphi} = p_!F_{\Ker, \varphi, \alpha}. \]
    For this, we follow the proof of~\cite[Proposition~5.17]{Raibaut}. We consider the canonical projections $\pi_1, \pi_2 \colon W[2n, 0, 0] \to W[n, 0, 0]$ and the definable morphism
    \[ \fonc{d}{W[2n, 0, 0]}{W[n, 0, 0],}{(w, z, y)}{(w, y - z).} \]
    By Proposition~\ref{prop:convol} and equality~\eqref{eq:convol-phi}, since the map $u_{\Phi_W}$ is $\Cons(W)\expe$-linear, the latter equality can be rewritten as
    \[ \inner{u_{\Phi_W}}{\varphi} = \LL^{\alpha n}\inner{u_{\Phi_W}}{\pi_{2!}(\pi_1^*\varphi \cdot d^*\cf_{\Ball_W(0, \alpha)^n})} \]
    and, since $\pi_2 = p \times \id_{h[n, 0, 0]}$, we get
    \[ \inner{u_{\Phi_W}}{\varphi} = \LL^{\alpha n}\inner{u_{\Phi_W}}{(p \times \id_{h[n, 0, 0]})_!(\pi_1^*\varphi \cdot d^*\cf_{\Ball_W(0, \alpha)^n})}. \]
    The function $\pi_1^*\varphi \cdot d^*\cf_{\Ball_W(0, \alpha)^n}$ of $\Sch_{W[n, 0, 0]}(W[2n, 0, 0])\expe$ is $p \times \id_{h[n, 0, 0]}$-convenient. Then by condition~\ref{it:pushforward}, the function $\inner{u_{\Phi_W \circ p}}{\pi_1^*\varphi \cdot d^*\cf_{\Ball_W(0, \alpha)^n}}$ is $p$-integrable and
    \[ \inner{u_{\Phi_W}}{\varphi} = \LL^{\alpha n}p_!\inner{u_{\Phi_W \circ p}}{\pi_1^*\varphi \cdot d^*\cf_{\Ball_W(0, \alpha)^n}}. \]
    The $\Cons(W)\expe$-linearity of the map $u_{\Phi_W}$ implies
    \[ \inner{u_{\Phi_W}}{\varphi} = \LL^{\alpha n}p_!(\varphi\inner{u_{\Phi_W \circ p}}{d^*\cf_{\Ball_W(0, \alpha)^n}}). \]
    For $i = 1, 2$, the diagram
    \begin{center}
        \begin{tikzcd}[column sep=2cm]
            & W[2n, 0, 0] \arrow[dl, to path=-| (\tikztotarget), rounded corners] \arrow[d] \arrow[r, "d"] & W[n, 0, 0] \arrow[d] \\
            W[n + n_i, 0, 0] \arrow[r, "q_i"] & W[2n_i, 0, 0] \arrow[r, "d_i"] & W[n_i, 0, 0]
        \end{tikzcd}
    \end{center}
    commutes where the morphism $q_i$ and the nameless arrows are the canonical projections. As pullbacks are ring morphisms, by definition~\eqref{eq:tensor-product}, we have
    \[ d^*\cf_{\Ball_W(0, \alpha)^n} = (d_1 \circ q_1)^*\cf_{\Ball_W(0, \alpha)^{n_1}} \otimes (d_2 \circ q_2)^*\cf_{\Ball_W(0, \alpha)^{n_2}}. \]
    With equality~\eqref{eq:noyau-mot} and the latter equation, we deduce that the function
    \begin{align*}
        F_{\Ker, \varphi, \alpha} &= \LL^{(\alpha \circ p)n}\varphi\inner{\Ker_{\Phi_W \circ p}(d_2 \circ q_2)^*\cf_{\Ball_W(0, \alpha)^{n_2}}}{(d_1 \circ q_1)^*\cf_{\Ball_W(0, \alpha)^{n_1}}} \\
        &= \LL^{(\alpha \circ p)n}\varphi\inner{u_{\Phi_W \circ p}}{(d_1 \circ q_1)^*\cf_{\Ball_W(0, \alpha)^{n_1}} \otimes (d_2 \circ q_2)^*\cf_{\Ball_W(0, \alpha)^{n_2}}} \\
        &= \LL^{(\alpha \circ p)n}\varphi\inner{u_{\Phi_W \circ p}}{d^*\cf_{\Ball_W(0, \alpha)^n}}
    \end{align*}
    is $p$-integrable and, thanks to the projection axiom (see~\cite[Theorem~4.1.1, Axiom~(A3)]{CluckersLoeser2010}), we get
    \begin{align*}
        \inner{u_{\Phi_W}}{\varphi} &= p_!F_{\Ker, \varphi, \alpha}
    \end{align*}
    which is equality~\eqref{eq:rel-u-K}.

    \sepproof

    We prove that the family $\Ker = (\Ker_{\Phi_W})_{\Phi_W \in \Def_P}$ is a kernel on $h[n_1, 0, 0] \times h[n_2, 0, 0]$. Let~$\Phi_W \colon W \to P$ and~$\Phi_{W'} \colon W' \to P$ be two objects and~$g \colon W \to W'$ be a morphism of~$\Def_P$.

    \subproof{Point~\ref{it:integrable}}
    Let $\varphi$ be a function of $\Sch_W(W[n, 0, 0])\expe$ and~$\alpha \colon W \to h[0, 0, 1]$ be a definable morphism such that
    \[ \varphi * \cf_{\Ball_W(0, \alpha)^n} = \LL^{-\alpha n}\varphi. \]
    Thanks to equation~\eqref{eq:rel-u-K}, we already seen that the function $F_{\Ker, \varphi, \alpha}$ is $p$-integrable. Furthermore, we assume that the function $\varphi$ is $g \times \id_{h[n, 0, 0]}$-convenient. By condition~\ref{it:pushforward}, the function $\inner{u_{\Phi_W}}{\varphi}$ is $g$-integrable. Consequently, by equality~\eqref{eq:rel-u-K}, the function~$p_!F_{\Ker, \varphi, \alpha}$ is also~$g$-integrable. By Fubini's theorem, we deduce that the function $F_{\Ker, \varphi, \alpha}$ is $g \circ p$-integrable.

    \subproof{Point~\ref{it:pullback-K}}
    Let $\varphi_1$ be a function of $\Sch_{W'}(W'[n_1, 0, 0])\expe$ and $\varphi_2$ be a function of~$\Sch_{W'}(W'[n_2, 0, 0])\expe$. Let us prove the equality
    \begin{equation}\label{eq:pullback-K-2}
        g^*\inner{\Ker_{\Phi_{W'}}\varphi_2}{\varphi_1} = \inner{\Ker_{\Phi_W}(g \times \id_{h[n_2, 0, 0]})^*\varphi_2}{(g \times \id_{h[n_1, 0, 0]})^*\varphi_1}.
    \end{equation}
    We denote by
    \begin{alignat*}{2}
        p_1 \colon &W[n, 0, 0] \to W[n_1, 0, 0], \qquad &p'_1 \colon &W'[n, 0, 0] \to W'[n_1, 0, 0], \\
        p_2 \colon &W[n, 0, 0] \to W[n_2, 0, 0], \qquad &p'_2 \colon &W'[n, 0, 0] \to W'[n_2, 0, 0].
    \end{alignat*}
    the canonical projections. For $i = 1, 2$, the diagram
    \begin{center}
        \begin{tikzcd}[column sep=3cm]
            W[n, 0, 0] \arrow[r, "g \times \id_{h[n, 0, 0]}"] \arrow[d, "p_i"'] & W'[n, 0, 0] \arrow[d, "p_i'"] \\
            W[n_i, 0, 0] \arrow[r, "g \times \id_{h[n_i, 0, 0]}"] & W'[n_i, 0, 0]
        \end{tikzcd}
    \end{center}
    commutes. As pullbacks are ring morphisms, with equality~\eqref{eq:noyau-mot}, we can write
    \begin{align*}
        g^*\inner{\Ker_{\Phi_{W'}}\varphi_2}{\varphi_1} &= g^*\inner{u_{\Phi_{W'}}}{\varphi_1 \otimes \varphi_2} \\
        &= \inner{u_{\Phi_W}}{(g \times \id_{h[n, 0, 0]})^*(\varphi_1 \otimes \varphi_2)} \\
        &= \inner{u_{\Phi_W}}{(g \times \id_{h[n, 0, 0]})^*(p_1'^*\varphi_1 \cdot p_2'^*\varphi_2)} \\
        &= \inner{u_{\Phi_W}}{(p_1' \circ (g \times \id_{h[n, 0, 0]}))^*\varphi_1 \cdot (p_2' \circ (g \times \id_{h[n, 0, 0]}))^*\varphi_2} \\
        &= \inner{u_{\Phi_W}}{((g \times \id_{h[n_1, 0, 0]}) \circ p_1)^*\varphi_1 \cdot ((g \times \id_{h[n_2, 0, 0]}) \circ p_2)^*\varphi_2} \\
        &= \inner{u_{\Phi_W}}{p_1^*((g \times \id_{h[n_1, 0, 0]})^*\varphi_1) \cdot p_2^*((g \times \id_{h[n_2, 0, 0]})^*\varphi_2)} \\
        &= \inner{u_{\Phi_W}}{(g \times \id_{h[n_1, 0, 0]})^*\varphi_1 \otimes (g \times \id_{h[n_2, 0, 0]})^*\varphi_2} \\
        &= \inner{\Ker_{\Phi_W}(g \times \id_{h[n_2, 0, 0]})^*\varphi_2}{(g \times \id_{h[n_1, 0, 0]})^*\varphi_1}
    \end{align*}
    which proves equality~\eqref{eq:pullback-K-2}.

    \subproof{Point~\ref{it:pushforward-K}}
    Let $\varphi_1$ be a function of $\Sch_W(W[n_1, 0, 0])\expe$ and let~$\varphi_2$ be a function of~$\Sch_W(W[n_2, 0, 0])\expe$. We assume that the function $\varphi_1 \otimes \varphi_2$ is~$g \times \id_{h[n, 0, 0]}$-convenient. We take the notation of condition~\ref{it:pushforward-K}. Let us prove that the function~$\inner{\Ker_{\Phi_W}\varphi_2}{\varphi_1}$ is~$g$-integrable and
    \[ g_!\inner{\Ker_{\Phi_W}\varphi_2}{\varphi_1} = p'_!F_{\Ker, \psi, \alpha} \]
    where $\psi \coloneq (g \times \id_{h[n, 0, 0]})_!(\varphi_1 \otimes \varphi_2)$. Since the function $\varphi_1 \otimes \varphi_2$ is~$g \times \id_{h[n, 0, 0]}$-convenient, by equation~\eqref{eq:noyau-mot}, the pushforward condition~\ref{it:pushforward} implies that the function
    \[ \inner{\Ker_{\Phi_W}\varphi_2}{\varphi_1} = \inner{u_{\Phi_W}}{\varphi_1 \otimes \varphi_2} \]
    is $g$-integrable and
    \[ g_!\inner{\Ker_{\Phi_W}\varphi_2}{\varphi_1} = \inner{u_{\Phi_{W'}}}{\psi}. \]
    Thus, relation~\eqref{eq:pushforward-K} follows from equation~\eqref{eq:rel-u-K}.

    \subproof{Point~\ref{it:independance-alpha-plus}}
    This point follows immediately from equality~\eqref{eq:rel-u-K}.

    \smallbreak

    This concludes the proof of the fact that the family $\Ker = (\Ker_{\Phi_W})_{\Phi_W \in \Def_P}$ is a kernel on $h[n_1, 0, 0] \times h[n_2, 0, 0]$.
\end{proof}

We also have the reciprocal of Theorem~\ref{thm:kernel-mot}.

\begin{theo}\label{thm:kernel-mot-rec}
    Let $\Ker = (\Ker_{\Phi_W})_{\Phi_W \in \Def_P}$ be a kernel on $h[n_1, 0, 0] \times h[n_2, 0, 0]$. Then there exists a unique distribution~$u$ of $\Sch_P'(P[n, 0, 0])\expe$ such that, for each object $\Phi_W \colon W \to P$ of~$\Def_P$ and all functions $\varphi_1$ of $\Sch_W(W[n_1, 0, 0])\expe$ and $\varphi_2$ of~$\Sch_W(W[n_2, 0, 0])\expe$, the equality
    \begin{equation}\label{eq:noyau-mot-rec}
        \inner{u_{\Phi_W}}{\varphi_1 \otimes \varphi_2} = \inner{\Ker_{\Phi_W}\varphi_2}{\varphi_1}
    \end{equation}
    holds. Moreover, for any object $\Phi_W \colon W \to P$ of $\Def_P$, for any function $\varphi$ of~$\Sch_W(W[n, 0, 0])\expe$ and for any definable morphism~$\alpha \colon W \to h[0, 0, 1]$ such that
    \[ \varphi * \cf_{\Ball_W(0, \alpha)^n} = \LL^{-\alpha n}\varphi, \]
    we have
    \begin{equation}\label{eq:rel-u-K-rec}
        \inner{u_{\Phi_W}}{\varphi} = p_!F_{\Ker, \varphi, \alpha}
    \end{equation}
    where we consider the canonical projection $p \colon W[n, 0, 0] \to W$.
\end{theo}

\begin{proof}
    Let $\Phi_W \colon W \to P$ be an object of $\Def_P$. For each function~$\varphi$ of~$\Sch_W(W[n, 0, 0])\expe$, we set $\inner{u_{\Phi_W}}{\varphi}$ as in equation~\eqref{eq:rel-u-K-rec}, that is
    \begin{equation}\label{eq:def-u}
        \inner{u_{\Phi_W}}{\varphi} \coloneq p_!(\LL^{(\alpha \circ p)n}\varphi\inner{\Ker_{\Phi_W \circ p}(d_2 \circ q_2)^*\cf_{\Ball_W(0, \alpha)^{n_2}}}{(d_1 \circ q_1)^*\cf_{\Ball_W(0, \alpha)^{n_1}}})
    \end{equation}
    for some definable morphism $\alpha \colon W \to h[0, 0, 1]$ such that $\varphi * \cf_{\Ball_W(0, \alpha)^n} = \LL^{-\alpha n}\varphi$. By condition~\ref{it:independance-alpha-plus}, this definition does not depend of the choice of the definable morphism~$\alpha$. We thus obtain a map
    \[ u_{\Phi_W} \colon \Sch_W(W[n, 0, 0])\expe \to \Cons(W)\expe. \]

    Let us show that the family $(u_{\Phi_W})_{W \in \Def_P}$ is a distribution which verifies equation~\eqref{eq:noyau-mot-rec}. Let $\Phi_W \colon W \to P$ and~$\Phi_{W'} \colon W' \to P$ be two objects and $g \colon W \to W'$ be a morphism of~$\Def_P$.

    \subproof{Equation~(\ref{eq:noyau-mot-rec})}
    Let $\varphi_1$ be a function of $\Sch_W(W[n_1, 0, 0])\expe$ and $\varphi_2$ be a function of $\Sch_W(W[n_2, 0, 0])\expe$. We want to prove the equation
    \[ \inner{u_{\Phi_W}}{\varphi_1 \otimes \varphi_2} = \inner{\Ker_{\Phi_W}\varphi_2}{\varphi_1}. \]
    We consider the Schwartz--Bruhat function $\varphi \coloneq \varphi_1 \otimes \varphi_2$. Let $\alpha \colon W \to h[0, 0, 1]$ be a definable morphism such that $\varphi * \cf_{\Ball_W(0, \alpha)^n} = \LL^{-\alpha n}\varphi$. By definition~\eqref{eq:def-u} and the projection axiom, we can write
    \begin{equation}\label{eq:def-u-2}
        \inner{u_{\Phi_W}}{\varphi} = p_!F_{\Ker, \varphi, \alpha}.
    \end{equation}
    We now use condition~\ref{it:pushforward-K} with $W = W'$ and $g = \id_W$. By definition, since
    \[ \id_W \times \id_{h[n, 0, 0]} = \id_{W[n, 0, 0]} \qquad \text{and} \qquad
       \id_{W[n, 0, 0]!} = \id_{\Cons(W{n, 0, 0})\expe}, \]
    the function $\varphi$ is $\id_W \times \id_{h[n, 0, 0]}$-convenient. Since $\id_{W!} = \id_{\Cons(W)\expe}$, we thus get
    \begin{equation}\label{eq:pushforward-K-id}
        \inner{\Ker_{\Phi_W}\varphi_2}{\varphi_1} = p_!F_{\Ker, \varphi, \alpha}.
    \end{equation}
    Thus equations~\eqref{eq:def-u-2} and~\eqref{eq:pushforward-K-id} lead to equation~\eqref{eq:noyau-mot-rec}.

    \subproof{Linearity}\label{proof:linearity}
    Let us prove that the map $u_{\Phi_W}$ is $\Cons(W)\expe$-linear. Let $\varphi$ and $\chi$ be two functions of~$\Sch_W(W[n, 0, 0])\expe$ and $\ell$ and $m$ be two functions of $\Cons(W)\expe$. Taking the maximum, we can assume that~$\alpha \coloneq \alpha^+(\varphi) = \alpha^+(\chi)$. Then the function $\ell \cdot \varphi + m \cdot \chi$ is a Schwartz--Bruhat function for which we can choose $\alpha^+(\ell \cdot \varphi + m \cdot \chi) = \alpha$. Recall that the $\Cons(W)\expe$-module structure on the set $\Sch_W(W[n, 0, 0])\expe$ gives
    \[ \ell \cdot \varphi + m \cdot \chi = p^*(\ell)\varphi + p^*(m)\chi. \]
    Then by the projection axiom and equality~\eqref{eq:def-u}, since the map $p_!$ is a group morphism, we obtain
    \begin{align*}
        &\inner{u_{\Phi_W}}{\ell \cdot \varphi + m \cdot \chi} 
        = \ell\inner{u_{\Phi_W}}{\varphi} + m\inner{u_{\Phi_W}}{\chi}.
    \end{align*}
    Thus, the map $u_{\Phi_W}$ is $\Cons(W)\expe$-linear.

    \subproof{Pullback condition~\ref{it:pullback}}
    Let $\varphi$ be a function of $\Sch_{W'}(W'[n, 0, 0])\expe$. Let us prove the equality
    \begin{equation}\label{eq:pullback-2}
        g^*\inner{u_{\Phi_{W'}}}{\varphi} = \inner{u_{\Phi_W}}{\psi}
    \end{equation}
    where $\psi \coloneq (g \times \id_{h[n, 0, 0]})^*\varphi$. By~\cite[Lemma~5.2]{Raibaut}, the function $\psi$ is a Schwartz--Bruhat function for which we can choose $\alpha^+(\psi) = \alpha^+(\varphi) \circ g$. By equality~\eqref{eq:def-u} and the projection axiom, the terms in equality~\eqref{eq:pullback-2} are
    \begin{align*}
        g^*\inner{u_{\Phi_{W'}}}{\varphi} &= g^*(\LL^{\alpha^+(\varphi)n}p'_!(\varphi\inner{\Ker_{\Phi_{W'} \circ p'}\tilde\varphi_2}{\tilde\varphi_1})) \\
        &= \LL^{\alpha^+(\psi)n}g^*p'_!(\varphi\inner{\Ker_{\Phi_{W'} \circ p'}\tilde\varphi_2}{\tilde\varphi_1})
    \end{align*}
    and
    \[ \inner{u_{\Phi_W}}{\psi} = \LL^{\alpha^+(\psi)n}p_!(\psi\inner{\Ker_{\Phi_W \circ p}\tilde\psi_2}{\tilde\psi_1}) \]
    with $\tilde\varphi_i \coloneq (d_i' \circ q_i')^*\cf_{\Ball_{W'}(0, \alpha^+(\varphi))^{n_i}}$ and $\tilde\psi_i \coloneq (d_i \circ q_i)^*\cf_{\Ball_W(0, \alpha^+(\psi))^{n_i}}$ for $i = 1, 2$. Then equality~\eqref{eq:pullback-2} is equivalent to the equality
    \begin{equation}\label{eq:pullback-3}
        p_!(\psi\inner{\Ker_{\Phi_W \circ p}\tilde\psi_2}{\tilde\psi_1}) = g^*p'_!(\varphi\inner{\Ker_{\Phi_{W'} \circ p'}\tilde\varphi_2}{\tilde\varphi_1}).
    \end{equation}

    Let us rewrite the left-hand side of equation~\eqref{eq:pullback-3}. Let $i$ be the index $1$ or $2$. We have the commutative diagram
    \begin{center}
        \begin{tikzcd}[column sep=3cm]
            W[n + n_i, 0, 0] \arrow[r, "g \times \id_{h[n + n_i, 0, 0]}"] \arrow[d, "d_i \circ q_i"'] & W'[n + n_i, 0, 0] \arrow[d, "d_i' \circ q_i'"] \\
            W[n_i, 0, 0] \arrow[r, "g \times \id_{h[n_i, 0, 0]}"] & W'[n_i, 0, 0]
        \end{tikzcd}
    \end{center}
    and, by definition of the function $\tilde\varphi_i$, we write
    \begin{align*}
        (g \times \id_{h[n + n_i, 0, 0]})^*\tilde\varphi_i &= (d_i' \circ q_i' \circ (g \times \id_{h[n + n_i, 0, 0]}))^*\cf_{\Ball_{W'}(0, \alpha^+(\varphi))^{n_i}} \\
        &= ((g \times \id_{h[n_i, 0, 0]}) \circ d_i \circ q_i)^*\cf_{\Ball_{W'}(0, \alpha^+(\varphi))^{n_i}} \\
        &= (d_i \circ q_i)^*\cf_{(g \times \id_{h[n_i, 0, 0]})^{-1}\Ball_{W'}(0, \alpha^+(\varphi))^{n_i}}.
    \end{align*}
    But we have the equality
    \[ (g \times \id_{h[n_i, 0, 0]})^{-1}\Ball_{W'}(0, \alpha^+(\varphi))^{n_i} = \Ball_W(0, \alpha^+(\varphi) \circ g)^{n_i} \]
    since, for all points $(p, x)$ of $W[n_i, 0, 0]$, we find
    \begin{align*}
        (p, x) \in (g \times \id_{h[n_i, 0, 0]})^{-1}\Ball_{W'}(0, \alpha^+(\varphi))^{n_i} &\iff (g(p), x) \in \Ball_{W'}(0, \alpha^+(\varphi))^{n_i} \\
        &\iff \ord x \geqslant \alpha^+(\varphi)(g(p)) \\
        &\iff \ord x \geqslant \alpha^+(\varphi) \circ g(p) \\
        &\iff (p, x) \in \Ball_W(0, \alpha^+(\varphi) \circ g)^{n_i}.
    \end{align*}
    As $\alpha^+(\psi) = \alpha^+(\varphi) \circ g$, we thus get
    \[ (g \times \id_{h[n + n_i, 0, 0]})^*\tilde\varphi_i = \tilde\psi_i. \]
    Then by equation~\eqref{eq:pullback-K} of condition~\ref{it:pullback-K}, the left-hand side of equation~\eqref{eq:pullback-3} can be rewritten as
    \[ p_!(\psi\inner{\Ker_{\Phi_W \circ p}\tilde\psi_2}{\tilde\psi_1}) = p_!(\psi(g \times \id_{h[n, 0, 0]})^*\inner{\Ker_{\Phi_{W'} \circ p'}\tilde\varphi_2}{\tilde\varphi_1}). \]

    We set
    \[ \tilde\varphi \coloneq \varphi\inner{\Ker_{\Phi_{W'} \circ p'}\tilde\varphi_2}{\tilde\varphi_1}. \]
    Then equality~\eqref{eq:pullback-3} is equivalent to the equality
    \[ p_!(\psi(g \times \id_{h[n, 0, 0]})^*\inner{\Ker_{\Phi_{W'} \circ p'}\tilde\varphi_2}{\tilde\varphi_1}) = g^*p'_!\tilde\varphi, \]
    that is
    \[ p_!(g \times \id_{h[n, 0, 0]})^*\tilde\varphi = g^*p'_!\tilde\varphi. \]
    But this relation follows from~\cite[Theorem~1.1]{CelyRaibaut} or equivalently from~\cite[Corollary~3.6.6]{CluckersHalupczok}. This concludes equality~\eqref{eq:pullback-2}.

    \subproof{Pushforward condition~\ref{it:pushforward}}
    Let $\varphi$ be a $g \times \id_{h[n, 0, 0]}$-convenient function of~$\Sch_W(W[n, 0, 0])\expe$. Let us show that the function $\inner{u_{\Phi_W}}{\varphi}$ is $g$-integrable and
    \begin{equation}\label{eq:pushforward-2}
        g_!\inner{u_{\Phi_W}}{\varphi} = \inner{u_{\Phi_{W'}}}{\psi}
    \end{equation}
    with $\psi \coloneq (g \times \id_{h[n, 0, 0]})_!\varphi$. By~\cite[Lemma~5.5]{Raibaut}, the function $\psi$ is a Schwartz--Bruhat function for which we can choose the morphism $\alpha^+(\psi)$ as~$\alpha^+(\varphi) = \alpha^+(\psi) \circ g$. By equality~\eqref{eq:def-u}, we then have
    \[ \inner{u_{\Phi_W}}{\varphi} = p_!F_{\Ker, \varphi, \alpha^+(\varphi)}. \]
    Since the function $\varphi$ is $g \times \id_{h[n, 0, 0]}$-convenient, by condition~\ref{it:integrable}, the function~$F_{\Ker, \varphi, \alpha^+(\varphi)}$ is $g \circ p$-integrable. We deduce, by Fubini's theorem, that the function~$\inner{u_{\Phi_W}}{\varphi}$ is~$g$-integrable.

    Let us prove equality~\eqref{eq:pushforward-2}. By equality~\eqref{eq:def-u}, the terms in equality~\eqref{eq:pushforward-2} are
    \[ g_!\inner{u_{\Phi_W}}{\varphi} = g_!(\LL^{\alpha^+(\varphi)n}p_!(\varphi\inner{\Ker_{\Phi_W \circ p}\tilde\varphi_2}{\tilde\varphi_1})) \]
    and
    \[ \inner{u_{\Phi_{W'}}}{\psi} = \LL^{\alpha^+(\psi)n}p'_!(\psi\inner{\Ker_{\Phi_{W'} \circ p'}\tilde\psi_2}{\tilde\psi_1}) \]
    with $\tilde\varphi_i \coloneq (d_i \circ q_i)^*\cf_{\Ball_W(0, \alpha^+(\varphi))^{n_i}}$ and $\tilde\psi_i \coloneq (d_i' \circ q_i')^*\cf_{\Ball_{W'}(0, \alpha^+(\psi))^{n_i}}$ for $i = 1, 2$. As in the last paragraph, since $\alpha^+(\varphi) = \alpha^+(\psi) \circ g$, by condition~\ref{it:pullback-K}, we can write
    \[ \inner{\Ker_{\Phi_W \circ p}\tilde\varphi_2}{\tilde\varphi_1} = (g \times \id_{h[n, 0, 0]})^*\inner{\Ker_{\Phi_{W'} \circ p'}\tilde\psi_2}{\tilde\psi_1}. \]
    Then equality~\eqref{eq:pushforward-2} is equivalent to the equality
    \[ g_!(\LL^{\alpha^+(\varphi)n}p_!(\varphi(g \times \id_{h[n, 0, 0]})^*\tilde\psi_0)) = \LL^{\alpha^+(\psi)n}p'_!(\psi\tilde\psi_0). \]
    with $\tilde\psi_0 \coloneq \inner{\Ker_{\Phi_W \circ p}\tilde\psi_2}{\tilde\psi_1}$. By the projection axiom, this latter equality is equivalent to the equality
    \[ g_!p_!(p^*(\LL^{\alpha^+(\varphi)n})\varphi(g \times \id_{h[n, 0, 0]})^*\tilde\psi_0) = p'_!(p'^*(\LL^{\alpha^+(\psi)n})\psi\tilde\psi_0), \]
    that is
    \[ g_!p_!(\LL^{(\alpha^+(\varphi) \circ p)n}\varphi(g \times \id_{h[n, 0, 0]})^*\tilde\psi_0) = p'_!(\LL^{(\alpha^+(\psi) \circ p')n}\psi\tilde\psi_0). \]
    Using the commutative diagram
    \begin{center}
        \begin{tikzcd}[column sep=2.5cm]
            W[n, 0, 0] \arrow[r, "g \times \id_{h[n, 0, 0]}"] \arrow[d, "p"'] & W'[n, 0, 0] \arrow[d, "p'"] \\
            W \arrow[r, "g"] & W'
        \end{tikzcd}
    \end{center}
    and Fubini's theorem, we have
    \[ g_! \circ p_! = (g \circ p)_! = (p' \circ (g \times \id_{h[n, 0, 0]}))_! = p'_! \circ (g \times \id_{h[n, 0, 0]})_!. \]
    Then it is enough to show the equality
    \[ (g \times \id_{h[n, 0, 0]})_!(\LL^{(\alpha^+(\varphi) \circ p)n}\varphi(g \times \id_{h[n, 0, 0]})^*\tilde\psi_0) = \LL^{(\alpha^+(\psi) \circ p')n}\psi\tilde\psi_0. \]
    By the projection axiom, the latter equality is equivalent to the equality
    \[ (g \times \id_{h[n, 0, 0]})_!(\LL^{(\alpha^+(\varphi) \circ p)n}\varphi) \cdot \tilde\psi_0 = \LL^{(\alpha^+(\psi) \circ p')n}\psi \cdot \tilde\psi_0. \]
    So it is enough to prove the equality
    \[ (g \times \id_{h[n, 0, 0]})_!(\LL^{(\alpha^+(\varphi) \circ p)n}\varphi) = \LL^{(\alpha^+(\psi) \circ p')n}\psi. \]
    Since $\psi = (g \times \id_{h[n, 0, 0]})_!\varphi$, by the projection axiom, the latter equality is equivalent to the equality
    \[ (g \times \id_{h[n, 0, 0]})_!(\LL^{(\alpha^+(\varphi) \circ p)n} \cdot \varphi) = (g \times \id_{h[n, 0, 0]})_!((g \times \id_{h[n, 0, 0]})^*\LL^{(\alpha^+(\psi) \circ p')n} \cdot \varphi). \]
    So it is enough to prove the equality
    \begin{equation}\label{eq:pushforward-3}
        \LL^{(\alpha^+(\varphi) \circ p)n} = (g \times \id_{h[n, 0, 0]})^*\LL^{(\alpha^+(\psi) \circ p')n}.
    \end{equation}
    As $\alpha^+(\varphi) = \alpha^+(\psi) \circ g$ and by the latter commutative diagram, we have
    \begin{align*}
        (g \times \id_{h[n, 0, 0]})^*\LL^{(\alpha^+(\psi) \circ p')n} &= \LL^{(\alpha^+(\psi) \circ p' \circ (g \times \id_{h[n, 0, 0]}))n} \\
        &= \LL^{(\alpha^+(\psi) \circ g \circ p)n} \\
        &= \LL^{(\alpha^+(\varphi) \circ p)n}.
    \end{align*}
    By going back over all these equivalent or sufficient conditions, this gives equality~\eqref{eq:pushforward-2}. This concludes the proof of the existence.

    \subproof{Uniqueness}
    Let us prove that a distribution $u$ of $\Sch_P'(P[n, 0, 0])\expe$ which satisfies equality~\eqref{eq:noyau-mot} is unique. Let $u$ and $u'$ be two distributions of $\Sch_P'(P[n, 0, 0])\expe$ such that, for each object $\Phi_W \colon W \to P$ of~$\Def_P$ and all functions $\varphi_1$ of $\Sch_W(W[n_1, 0, 0])\expe$ and $\varphi_2$ of $\Sch_W(W[n_2, 0, 0])\expe$, we have
    \begin{align}
        \inner{u_{\Phi_W}}{\varphi_1 \otimes \varphi_2} &= \inner{\Ker_{\Phi_W}\varphi_2}{\varphi_1} \qquad \text{and} \label{eq:noyau-mot-u} \\
        \inner{u'_{\Phi_W}}{\varphi_1 \otimes \varphi_2} &= \inner{\Ker_{\Phi_W}\varphi_2}{\varphi_1}. \label{eq:noyau-mot-u'}
    \end{align}
    Let $\Phi_W \colon W \to P$ be an object of~$\Def_P$. Let $\varphi$ be a function of $\Sch_W(W[n, 0, 0])\expe$. With equalities~\eqref{eq:noyau-mot-u} and~\eqref{eq:noyau-mot-u'}, we know that the two distributions $u$ and $u'$ satisfy respectively equality~\eqref{eq:rel-u-K-rec}, that is
    \begin{align*}
        \inner{u_{\Phi_W}}{\varphi} &= p_!F_{\Ker, \varphi, \alpha^+(\varphi)} \qquad \text{and} \\
        \inner{u'_{\Phi_W}}{\varphi} &= p_!F_{\Ker, \varphi, \alpha^+(\varphi)}.
    \end{align*}
    We deduce that $\inner{u_{\Phi_W}}{\varphi} = \inner{u'_{\Phi_W}}{\varphi}$ which concludes $u = u'$.
\end{proof}

\begin{rema}\label{rem:discussion-hypothesis}
    Let us discuss about conditions~\ref{it:integrable}--\ref{it:independance-alpha-plus}. The first condition ensures integrability in order to consider the functions $F_{\Ker, \varphi, \alpha}$. Conditions~\ref{it:pullback-K} and~\ref{it:pushforward-K} are pretty natural: we want that the distribution satisfies the pullback condition~\ref{it:pullback} and the pushforward condition~\ref{it:pushforward}.

    Thanks to the proof of Theorem~\ref{thm:kernel-mot-rec}, condition~\ref{it:pushforward-K} is satisfied for any definable morphism $g$ of $\Def_P$ if and only if it is satisfied for any identity morphism $g = \id_W$. Indeed, the direct sense is clear. Reciprocally, if a kernel $\Ker$ satisfies the condition~\ref{it:pushforward-K} for definable morphisms of the type $g = \id_W$, then by Theorem~\ref{thm:kernel-mot-rec}, it defines a distribution $u$ and, by Theorem~\ref{thm:kernel-mot}, its kernel, namely the kernel $\Ker$ by uniqueness, satisfies the condition~\ref{it:pushforward-K} for general definable morphisms $g$.

    Condition~\ref{it:independance-alpha-plus} is implied by condition~\ref{it:pushforward-K} for Schwartz--Bruhat functions of the form $\varphi = \varphi_1 \otimes \varphi_2$. Nevertheless, it does not seems to be true for general Schwartz--Bruhat functions. It would be the case if one could compare the morphisms $\alpha$ and $\beta$, but it is not the case. Indeed, a definable morphism is a family of functions indexed by the field of $\Field_k$ and, for instance, it could be%
        \footnote{Actually, it can happen even if $W = \{*\}$. Indeed, let us take $k = \QQ$. We consider the definable morphism $\alpha \colon \{*\} \to \{0, 1\} \subset h[0, 0, 1]$ defined by the equality
        \[ \alpha(*, K) = \begin{cases}
            1 &\text{if } \exists z \in K, z^2 = 2, \\
            0 &\text{else}
        \end{cases} \]
        for all fields $K$ of $\Field_\QQ$. We also consider the definable morphism $\beta \coloneq 1 - \alpha$. Then
        \[ \alpha(*, \QQ) = 0 < \beta(*, \QQ) = 1 \qquad \text{and} \qquad
           \alpha(*, \QQ(\sqrt2)) = 1 > \beta(*, \QQ(\sqrt2)) = 0. \]}
    that there exists two points $w$ and~$w'$ of~$W$ such that
    \[ \alpha(w) < \beta(w) \qquad \text{and} \qquad
       \alpha(w') > \beta(w'). \]

    This hypothesis~\ref{it:independance-alpha-plus} allows to define the distribution $u$ with formula~\eqref{eq:rel-u-K-rec}. Moreover, it provides flexibility in the choice of morphisms~$\alpha^+(-)$. For instance, if we work with a sum~$\varphi = \varphi_1 + \dots + \varphi_\ell$ of some Schwartz--Bruhat functions $\varphi_1$, \dots, $\varphi_\ell$, we want to take the morphism~$\alpha^+(\varphi)$ as the morphism~$\max(\alpha^+(\varphi_1), \dots, \alpha^+(\varphi_\ell))$ and not the optimal morphism to simplify computations (it is the case when we verify p.~\pageref{proof:linearity} the linearity of the maps $u_{\Phi_W}$).
\end{rema}

\begin{exem}
    Let $u$ be a $P$-locally integrable function on $P[n, 0, 0]$, that is a function of $\Cons(P[n, 0, 0])\expe$ such that, for any definable morphism $\alpha \colon P \to h[0, 0, 1]$, the function $u\cf_{\Ball_P(0, \alpha)^n}$ is $P$-integrable. By~\cite[Example~5.14]{Raibaut}, it induces a distribution~$u$ on $P[n, 0, 0]$ defined in the following way: for each object $\Phi_W \colon W \to P$ of $\Def_P$ and each function $\varphi$ of $\Sch_W(W[n, 0, 0])\expe$, if we denote by $p \colon W[n, 0, 0] \to W$ the canonical projection, then the function $(\Phi_W \times \id_{h[n, 0, 0]})^*u \cdot \varphi$ is $p$-integrable and we set
    \[ \inner{u_{\Phi_W}}{\varphi} \coloneq p_!((\Phi_W \times \id_{h[n, 0, 0]})^*u \cdot \varphi). \]
    We denote by $\Ker$ its kernel. Let $p_i \colon W[n, 0, 0] \to W[n_i, 0, 0]$ be the canonical projection for $i = 1, 2$ and $q_1 \colon W[n_1, 0, 0] \to W$ be the canonical projection. Let~$\Phi_W \colon W \to P$ be an object of $\Def_P$. Let $\varphi_1$ be a function of $\Sch_W(W[n_1, 0, 0])\expe$ and $\varphi_2$ be a function of~$\Sch_W(W[n_2, 0, 0])\expe$. Then by Fubini's theorem and the projection axiom, the function $\inner{\Ker_{\Phi_W}\varphi_2}{\varphi_1}$ is given by the equality
    \[ \inner{\Ker_{\Phi_W}\varphi_2}{\varphi_1} = q_{1!}(p_{1!}((\Phi_W \times \id_{h[n, 0, 0]})^*u \cdot p_2^*\varphi_2) \cdot \varphi_1). \]
    Using the same notation as in Remark~\ref{rem:informal}, the function $p_{1!}((\Phi_W \times \id_{h[n, 0, 0]})^*u \cdot p_2^*\varphi_2)$ can be interpreted as the function
    \[ (w, x) \longmapsto \int_{h[n_2, 0, 0]} u(\Phi_W(w), x, y)\varphi_2(w, y) \dd y. \]
    This is the motivic analogue of Example~\ref{exem:locally-integrable-kernel}.
\end{exem}

\subsection{Wave front sets of kernels}

We study now the microlocally smooth definable data and wave front sets of kernels induced by distributions.

From now on, we take $P = \{*\}$. For a distribution $u$ of~$\Sch'(h[n, 0, 0])\expe$ (resp. a kernel $\Ker$ on $h[n_1, 0, 0] \times h[n_2, 0, 0]$) and a definable subassignment~$W$, we consider the final morphism~$\Phi_W \colon W \to \{*\}$ and we set
\[ u_W \coloneq u_{\Phi_W} \qquad
   (\text{resp. } \Ker_W \coloneq \Ker_{\Phi_W}). \]

\subsubsection{Induced distributions}

We first need to define a notion of induced distributions by a distribution defined on a product space and a Schwartz--Bruhat function. This is the motivic analog of the distributions $\Ker\psi$ of §\ref{sec:p}. We keep the same notation as~§\ref{ssec:Schwartz}.

\begin{prop}\label{prop:induced}
    Let $u$ be a distribution of $\Sch'(h[n, 0, 0])\expe$ and $\Ker$ be its kernel. Let $\varphi_2$ be a function of $\Sch(h[n_2, 0, 0])\expe$. Then for each definable subassignment $W$, we consider the canonical projection $\pi_W \colon W[n_2, 0, 0] \to h[n_2, 0, 0]$, the function
    \[ \varphi_{2, W} \coloneq \pi_W^*\varphi_2 \]
    belongs to $\Sch_W(W[n_2, 0, 0])\expe$ and we set
    \begin{equation}\label{eq:induced}
        u_{\varphi_2, W} \coloneq \Ker_W\varphi_{2, W}.
    \end{equation}
    Moreover, the family $(u_{\varphi_2, W})_{W \in \Def_k}$ is a distribution on $h[n_1, 0, 0]$, denoted by $u_{\varphi_2}$.
\end{prop}

\begin{proof}
    For each definable subassignment $W$, since $\pi_W = \Phi_W \times \id_{h[n_2, 0, 0]}$ where we consider the final morphism $\Phi_W \colon W \to \{*\}$, by~\cite[Lemma~5.2]{Raibaut}, the function $\varphi_{2, W}$ belongs to $\Sch_W(W[n_2, 0, 0])\expe$.

    Let $g \colon W \to W'$ be a definable morphism. Let us prove the pullback condition~\ref{it:pullback}. Let $\varphi_1$ be a function of $\Sch_{W'}(W'[n_1, 0, 0])\expe$. Thanks to condition~\ref{it:pullback-K}, we can write
    \begin{align*}
        g^*\inner{u_{\varphi_2, W'}}{\varphi_1} &= g^*\inner{\Ker_{W'}\varphi_{2, W'}}{\varphi_1} \\
        &= \inner{\Ker_W(g \times \id_{h[n_2, 0, 0]})^*\varphi_{2, W'}}{(g \times \id_{h[n_1, 0, 0]})^*\varphi_1} \\
        &= \inner{\Ker_W[\pi_{W'} \circ (g \times \id_{h[n_2, 0, 0]})]^*\varphi_2}{(g \times \id_{h[n_1, 0, 0]})^*\varphi_1} \\
        &= \inner{\Ker_W\pi_W^*\varphi_2}{(g \times \id_{h[n_1, 0, 0]})^*\varphi_1} \\
        &= \inner{u_{\varphi_2, W}}{(g \times \id_{h[n_1, 0, 0]})^*\varphi_1}
    \end{align*}
    which gives condition~\ref{it:pullback}.

    Let us prove the pushforward condition~\ref{it:pushforward}. Let $\varphi_1$ be a $g \times \id_{h[n_1, 0, 0]}$-convenient function of $\Sch_W(W[n_1, 0, 0])\expe$. Then the function $\varphi_1 \otimes \varphi_{2, W}$ is $g \times \id_{h[n, 0, 0]}$-convenient since the function $\varphi_{2, W}$ does not depend of the first variable, that is the one relating to the definable subassignment $W$. Moreover, with evaluation~\cite{CluckersHalupczok}, we get the equality
    \[ (g \times \id_{h[n, 0, 0]})_!(\varphi_1 \otimes \varphi_{2, W}) = (g \times \id_{h[n_1, 0, 0]})_!\varphi_1 \otimes \varphi_{2, W'} \]
    since, with integral notation, we have
    \[ \int_{g^{-1}(w')} \varphi_1(w, x)\varphi_2(y) \dd w = \int_{g^{-1}(w')} \varphi_1(w, x) \dd w \cdot \varphi_2(y) \]
    for all points $(w', x, y)$ of $W'[n, 0, 0]$. Then by definition~\eqref{eq:induced} and the pushforward condition~\ref{it:pushforward} for the distribution $u$, the function
    \[ \inner{u_{\varphi_2, W}}{\varphi_1} = \inner{\Ker_W\varphi_{2, W}}{\varphi_1} = \inner{u_W}{\varphi_1 \otimes \varphi_{2, W}} \]
    is $g$-integrable and we can write
    \begin{align*}
        g_!\inner{u_{\varphi_2, W}}{\varphi_1} &= g_!\inner{u_W}{\varphi_1 \otimes \varphi_{2, W}} \\
        &= \inner{u_{W'}}{(g \times \id_{h[n, 0, 0]})_!(\varphi_1 \otimes \varphi_{2, W})} \\
        &= \inner{u_{W'}}{(g \times \id_{h[n_1, 0, 0]})_!\varphi_1 \otimes \varphi_{2, W'}} \\
        &= \inner{\Ker_{W'}\varphi_{2, W'}}{(g \times \id_{h[n_1, 0, 0]})_!\varphi_1} \\
        &= \inner{u_{\varphi_2, W'}}{(g \times \id_{h[n_1, 0, 0]})_!\varphi_1}
    \end{align*}
    which gives condition~\ref{it:pushforward} and concludes the proof.
\end{proof}

\subsubsection{Microlocally smooth definable data and kernels}

We recall the notion of microlocally smooth definable data for a distribution introduced by Raibaut~\cite{Raibaut}. We endow the definable subassignment $h[0, 0, 1]$ with the discrete topology.

Let $n$ and $\n$ be two positive integers. We consider the definable subassignments
\[ \Lambda_\n \coloneq \{x \in h[1, 0, 0] \mid \ord x \equiv 0 \mod\n, \; \ac x = 1\} \]
and
\[ \B_\n^n \coloneq \{\xi \in h[n, 0, 0] \mid 0 \leqslant \ord\xi < \n\} \]
where we define $\ord\xi \coloneq \min_{1 \leqslant i \leqslant n} \ord\xi_i$ for each point $\xi = (\xi_1, \dots, \xi_n)$ of $h[n, 0, 0]$.

\begin{defi}[Raibaut~{\cite[Definition~6.6]{Raibaut}}]
    A \emph{$\Lambda_\n$-microlocally smooth definable data} of a distribution $u$ of $\Sch'(h[n, 0, 0])\expe$ is the data of
    \begin{itemize}
         \item a definable subassignment $\A$ of $h[n, 0, 0] \times \B_\n^n$;

         \item two continuous definable morphisms $\r, \check\r \colon \A \to h[0, 0, 1]$ such that, for all points~$(x_0, \xi_0, K)$ of $\A$, we have
         \begin{equation}\label{eq:inclusion-A}
            \Ball(x_0, \r_K(x_0, \xi_0)) \times \Ball(\xi_0, \check\r_K(x_0, \xi_0)) \subset \A(K)
                \footnote{This inclusion will simply be denoted by $\Ball(x_0, \r(x_0, \xi_0)) \times \Ball(\xi_0, \check \r(x_0, \xi_0)) \subset \A$.}
         \end{equation}
         and $\check\r_K(-, \xi_0) \geqslant \n$;

         \item a continuous definable morphism $N \colon B \to h[0, 0, 1]$ with
         \[ B \coloneq \{((x_0, \xi_0), x', r') \in \A \times h[n, 0, 0] \times h[0, 0, 1] \mid \Ball(x', r') \subset \Ball(x_0, \r(x_0, \xi_0))\} \]
    \end{itemize}
    such that
    \begin{equation}\label{eq:msdd}
        \inner{u_{\Lambda_\n \times D}}{T}\cf_E\cf_{B_N} = \inner{u_{\Lambda_\n \times D}}{T}\cf_E
    \end{equation}
    where we consider the definable subassignments
    \begin{align*}
        D &\coloneq h[3n, 0, 0] \times h[0, 0, 1] \times h[n, 0, 0], \\
        E &\coloneq \{(\lambda, (x_0, \xi_0), x', r', \xi) \in \Lambda_\n \times B \times h[n, 0, 0] \mid \xi \in \Ball(\xi_0, \check\r(x_0, \xi_0))\}, \\
        B_N &\coloneq \{(\lambda, (x_0, \xi_0), x', r', \xi) \in \Lambda_\n \times B \times h[n, 0, 0] \mid \ord\lambda \geqslant N((x_0, \xi_0), x', r')\}
    \end{align*}
    and the function $T$ of $\Sch_{\Lambda_\n \times D}(\Lambda_\n \times D[n, 0, 0])\expe$ defined by the equality
    \[ T(\lambda, (x_0, \xi_0), x', r', \xi, x) \coloneq \cf_{\Ball(x', r')}(x)\Exp(\inner{x}{\lambda\xi}) \]
    for all points $(\lambda, (x_0, \xi_0), x', r', \xi, x)$ of $\Lambda_\n \times D[n, 0, 0]$.
\end{defi}

\begin{rema}\label{rem:complementary}
    Taking the complementary in $\Lambda_\n \times B \times h[n, 0, 0]$, equation~\eqref{eq:msdd} can be rewritten as
    \begin{equation}\label{eq:msdd-comp}
        \inner{u_{\Lambda_\n \times D}}{T}\cf_E\cf_{(\Lambda_\n \times B \times h[n, 0, 0]) \setminus B_N} = 0.
    \end{equation}
\end{rema}

We now state the motivic analog of Theorem~\ref{thm:WF-p} which concerns $\Lambda_\n$-microlocally smooth data of the induced distribution. The result on wave front sets will be done in Corollary~\ref{coro:WF-mot}. We always keep the same notation as~§\ref{ssec:Schwartz}.

\begin{theo}\label{thm:WF-mot}
    Let $u$ be a distribution of $\Sch'(h[n, 0, 0])\expe$ and~$(\A, \r, \check\r, N)$ be a~$\Lambda_n$-microlocally smooth definable data of the distribution $u$. Let $\varphi_2$ be a function of~$\Sch(h[n_2, 0, 0])\expe$ and let $u_{\varphi_2}$ be the induced distribution of Proposition~\ref{prop:induced}. We consider definable morphisms $\alpha^-(\varphi_2)$ and $\alpha^+(\varphi_2)$ as in Definition~\ref{def:SB}. Then the quadruplet $(\A', \r', \check\r', N')$ defined by the equalities
    \begin{align*}
        \A' &\coloneq \{(x_0, \xi_0) \in h[n_1, 0, 0] \times \B_\n^{n_1} \mid \forall y \in \Ball(0, \alpha^-(\varphi_2))^{n_2}, \; (x_0, y, \xi_0, 0) \in \A\}, \\
        \r'(x_0, \xi_0) &\coloneq \max(\alpha^+(\varphi_2), \max(\r(x_0, y, \xi_0, 0) : y \in \Ball(0, \alpha^-(\varphi_2))^{n_2})), \\
        \check\r'(x_0, \xi_0) &\coloneq \max(\check\r(x_0, y, \xi_0, 0) : y \in \Ball(0, \alpha^-(\varphi_2))^{n_2})
    \end{align*}
    for all points $(x_0, \xi_0)$ of $\A'$ and
    \[ N'((x_0, \xi_0), x', r') \coloneq \min(N((x_0, y, \xi_0, 0), (x', y), r') : y \in \Ball(0, \alpha^-(\varphi_2))^{n_2}) \]
    for all points $((x_0, \xi_0), x', r')$ of
    \[ B' \coloneq \{((x_0, \xi_0), x', r') \in \A' \times h[n, 0, 0] \times h[0, 0, 1] \mid \Ball(x', r') \subset \Ball(x_0, \r'(x_0, \xi_0))\} \]
    is a $\Lambda_\n$-microlocally smooth definable data of the distribution $u_{\varphi_2}$ where the considered maxima and minimum are well defined in virtue of~\cite[Proposition~3.3]{Raibaut}.
\end{theo}

\begin{proof}
    First, let us verify that the definable morphism $N'$ is well defined. We set
    \begin{multline*}
        B \coloneq \{((x_0, y_0, \xi_0, \eta_0), (x', y'), r') \in \A \times h[n, 0, 0] \times h[0, 0, 1] \\
        \mid \Ball((x', y'), r') \subset \Ball((x_0, y_0), \r(x_0, y_0, \xi_0, \eta_0))\}.
    \end{multline*}
    Let~$((x_0, \xi_0), x', r')$ be a point of $B'$. First, we prove that, for any point $y$ of~$\Ball(0, \alpha^-(\varphi_2))^{n_2}$, we have
    \[ ((x_0, y, \xi_0, 0), (x', y), r') \in B. \]
    Let $y$ be a point of $\Ball(0, \alpha^-(\varphi_2))^{n_2}$. Then the definition of the definable subassignment~$\A'$ gives $(x_0, y, \xi_0, 0) \in \A$. It remains to verify that
    \[ \Ball((x', y), r') \subset \Ball((x_0, y), \r(x_0, y, \xi_0, 0)), \]
    that is
    \begin{align*}
        \Ball(x', r') &\subset \Ball(x_0, \r(x_0, y, \xi_0, 0)) \qquad \text{and} \qquad \\
        r' &\geqslant \r(x_0, y, \xi_0, 0).
    \end{align*}
    But since $((x_0, \xi_0), x', r') \in B'$ and $\r'(x_0, \xi_0) \geqslant \r(x_0, y, \xi_0, 0)$, we know that
    \[ \Ball(x', r') \subset \Ball(x_0, \r'(x_0, \xi_0)) \subset \Ball(x_0, \r(x_0, y, \xi_0, 0)) \]
    and thus $r' \geqslant \r(x_0, y, \xi_0, 0)$. This proves that $((x_0, y, \xi_0, 0), (x', y), r') \in B$. Now, the quantity $N'((x_0, \xi_0), x', r')$ for a point $((x_0, \xi_0), x', r')$ of $B'$ is well-defined by~\cite[Proposition~3.3]{Raibaut} and it gives a definable morphism $N'$.

    \bigbreak

    Let us prove inclusion~\eqref{eq:inclusion-A} for the quadruplet $(\A', \r', \check\r', N')$, namely we want to show that, for any point $(x_0, \xi_0)$ of~$\A'$, we have
    \begin{equation}\label{eq:inclusion-A'}
        \Ball(x_0, \r'(x_0, \xi_0)) \times \Ball(\xi_0, \check\r'(x_0, \xi_0)) \subset \A'.
    \end{equation}
    Let $(x_0, \xi_0)$ be a point of $\A'$. By definition of the definable subassignment $\A'$, for all points $y$ of $\Ball(0, \alpha^-(\varphi_2))^{n_2}$, we have $(x_0, y, \xi_0, 0) \in \A$ and so
    \begin{equation}\label{eq:inclusion-A-2}
        \Ball((x_0, y), \r(x_0, y, \xi_0, 0)) \times \Ball((\xi_0, 0), \check\r(x_0, y, \xi_0, 0)) \subset \A.
    \end{equation}
    Let $(x, \xi)$ be a point of $\Ball(x_0, \r'(x_0, \xi_0)) \times \Ball(\xi_0, \check\r'(x_0, \xi_0))$. Let us prove that $(x, \xi) \in \A'$. Let $y$ be a point of~$\Ball(0, \alpha^-(\varphi_2))^{n_2}$. We need to show that $(x, y, \xi, 0) \in \A$. On the one hand, we get
    \[ (x, y) \in \Ball((x_0, y), \r'(x_0, \xi_0)) \subset \Ball((x_0, y), \r(x_0, y, \xi_0, 0)) \]
    since $x \in \Ball(x_0, \r'(x_0, \xi_0))$ and $\r'(x_0, \xi_0) \geqslant \r(x_0, y, \xi_0, 0)$. On the other hand, the same argument gives
    \[ (\xi, 0) \in \Ball((\xi_0, 0), \check\r(x_0, y, \xi_0, 0)). \]
    By inclusion~\eqref{eq:inclusion-A-2}, this gives $(x, y, \xi, 0) \in \A$. We have proven inclusion~\eqref{eq:inclusion-A'}. Furthermore, we get $\check\r'(-, \xi_0) \geqslant \n$.

    \bigbreak

    Moreover, by the proof of~\cite[Lemma~3.1]{Raibaut}, the definable morphisms $\r'$, $\check\r'$ and $N'$ are continuous.

    \bigbreak

    We must now prove equality~\eqref{eq:msdd} for the quadruplet $(\A', \r', \check\r', N')$ and that will conclude the proof. Namely, if we consider the definable subassignments
    \begin{align*}
        D' &\coloneq h[3n_1, 0, 0] \times h[0, 0, 1] \times h[n_1, 0, 0], \\
        E' &\coloneq \{(\lambda, (x_0, \xi_0), x', r', \xi) \in \Lambda_\n \times B' \times h[n_1, 0, 0] \mid \xi \in \Ball(\xi_0, \check\r'(x_0, \xi_0))\}, \\
        B_{N'} &\coloneq \{(\lambda, (x_0, \xi_0), x', r', \xi) \in \Lambda_\n \times B' \times h[n_1, 0, 0] \mid \ord\lambda \geqslant N'((x_0, \xi_0), x', r')\}
    \end{align*}
    and the function $T'$ of $\Sch_{\Lambda_\n \times D'}(\Lambda_\n \times D'[n_1, 0, 0])\expe$ defined by the equality
    \[ T'(\lambda, (x_0, \xi_0), x', r', \xi, x) \coloneq \cf_{\Ball(x', r')}(x)\Exp(\inner{x}{\lambda\xi}) \]
    for all points $(\lambda, (x_0, \xi_0), x', r', \xi, x)$ of $\Lambda_\n \times D'[n_1, 0, 0]$, we want to prove the equality
    \begin{equation}\label{eq:msdd-2}
        \inner{u_{\varphi_2, \Lambda_\n \times D'}}{T'}\cf_{E'}\cf_{(\Lambda_\n \times B' \times h[n_1, 0, 0]) \setminus B_{N'}} = 0
    \end{equation}
    using Remark~\ref{rem:complementary}.

    By the definition~\eqref{eq:induced} of the distribution $u_{\varphi_2}$ and the~$\Cons(\Lambda_\n \times D')\expe$-linearity of the map $u_{\Lambda_\n \times D'}$, the left member of equation~\eqref{eq:msdd-2} is
    \begin{multline*}
        \inner{u_{\varphi_2, \Lambda_\n \times D'}}{T'}\cf_{E'}\cf_{(\Lambda_\n \times B' \times h[n_1, 0, 0]) \setminus B_{N'}} \\
        = \inner{u_{\Lambda_\n \times D'}}{T' \otimes \pi_{\Lambda_\n \times D'}^*\varphi_2 \cdot p'^*(\cf_{E'}\cf_{(\Lambda_\n \times B' \times h[n_1, 0, 0]) \setminus B_{N'}})}.
    \end{multline*}
    In order to apply condition~\ref{it:pushforward}, we first rewrite (see equation~\eqref{eq:function-eq}) the function
    \begin{equation}\label{eq:function}
        T' \otimes \pi_{\Lambda_\n \times D'}^*\varphi_2 \cdot p'^*(\cf_{E'}\cf_{(\Lambda_\n \times B' \times h[n_1, 0, 0]) \setminus B_{N'}}).
    \end{equation}
    using a convolution identity~\eqref{eq:locally-constant}. For that, we consider the function $f$ on $\Lambda_\n \times D'[n, 0, 0]$ defined by the equality
    \[ f(\lambda, (x_0, \xi_0), x', r', \xi, (z_1, z_2)) \coloneq \cf_{\Ball(0, r')^{n_1}}(z_1) \]
    for all points $(\lambda, (x_0, \xi_0), x', r', \xi, (z_1, z_2))$ of $\Lambda_\n \times D'[n, 0, 0]$ and the functions $g$ and $h$ on~$\Lambda_\n \times D'[n, 0, 0][n, 0, 0]$ defined by the equalities
    \begin{align*}
        g(\lambda, (x_0, \xi_0), x', r', \xi, (z_1, z_2), (x, y)) &\coloneq \cf_{\Ball(0, r')^{n_2}}(y - z_2) = \cf_{\Ball(y, r')}(z_2) \qquad \text{and} \\
        h(\lambda, (x_0, \xi_0), x', r', \xi, (z_1, z_2), (x, y)) &\coloneq \cf_{\Ball(0, \alpha^-(\varphi_2))^{n_2}}(z_2)
    \end{align*}
    for all points $(\lambda, (x_0, \xi_0), x', r', \xi, (z_1, z_2), (x, y))$ of $\Lambda_\n \times D'[n, 0, 0][n, 0, 0]$. Moreover, we denote by
    \begin{itemize}
        \item $\pi_1 \colon \Lambda_\n \times D'[n, 0, 0][n, 0, 0] \to \Lambda_\n \times D'[n, 0, 0]$ the projection
        \[ (\lambda, (x_0, \xi_0), x', r', \xi, (z_1, z_2), (x, y)) \longmapsto (\lambda, (x_0, \xi_0), x', r', \xi, (z_1, z_2)), \]

        \item $\pi_2 \colon \Lambda_\n \times D'[n, 0, 0][n, 0, 0] \to \Lambda_\n \times D'[n, 0, 0]$ the projection
        \[ (\lambda, (x_0, \xi_0), x', r', \xi, (z_1, z_2), (x, y)) \longmapsto (\lambda, (x_0, \xi_0), x', r', \xi, (x, y)), \]

        \item $\rho \colon \Lambda_\n \times D'[n, 0, 0] \to \Lambda \times D'[n_1, 0, 0]$ the projection
        \[ (\lambda, (x_0, \xi_0), x', r', \xi, (x, y)) \longmapsto (\lambda, (x_0, \xi_0), x', r', \xi, x), \]

        \item $\sigma \colon \Lambda_\n \times D'[n, 0, 0] \to h[n_2, 0, 0]$ the projection
        \[ (\lambda, (x_0, \xi_0), x', r', \xi, (z_1, z_2)) \longmapsto z_2, \]

        \item $p' \colon \Lambda_\n \times D'[n, 0, 0] \to \Lambda_\n \times D'$ the projection
        \[ (\lambda, (x_0, \xi_0), x', r', \xi, (x, y)) \longmapsto (\lambda, (x_0, \xi_0), x', r', \xi). \]
    \end{itemize}
    These projections can be sum up with the following diagram.
    \begin{center}
        \begin{tikzcd}[column sep=2cm]
            && \Lambda_\n \times D' \\
            \Lambda_\n \times D'[2n, 0, 0]
                \arrow[r, "\pi_1", shift left=2]
                \arrow[r, "\pi_2"', shift right=2] &
            \Lambda_\n \times D'[n, 0, 0]
                \arrow[ur, "p'", to path={[pos=.75] |- (\tikztotarget) \tikztonodes}, rounded corners]
                \arrow[dr, "\rho"', to path={[pos=.75] |- (\tikztotarget) \tikztonodes}, rounded corners]
                \arrow[r, "\sigma"]
            & h[n_2, 0, 0] \\
            && \Lambda_\n \times D'[n_1, 0, 0]
        \end{tikzcd}
    \end{center}
    Let $(\lambda, (x_0, \xi_0), x', r', \xi, (x, y))$ be a point of $\Lambda_\n \times D'[n, 0, 0]$. We consider the final morphism $\Phi_{\Lambda_\n \times D'[n_1, 0,0]} \colon \Lambda_\n \times D'[n_1, 0,0] \to \{*\}$. Since $\sigma = \Phi_{\Lambda_\n \times D'[n_1, 0,0]} \times \id_{h[n_2, 0, 0]}$ and $\varphi_2 \in \Sch(h[n_2, 0, 0])\expe$, by~\cite[Lemma~5.2]{Raibaut}, we get
    \[ \sigma^*\varphi_2 \in \Sch_{\Lambda_\n \times D'[n_1, 0,0]}(\Lambda_\n \times D'[n, 0,0])\expe. \]
    Then by definition~\eqref{eq:tensor-product}, identities~\eqref{eq:locally-constant} and~\eqref{eq:bounded-support} of Definition~\ref{def:SB} for the function~$\sigma^*\varphi_2$, the Fubini's theorem, the projection axiom, the inequality~$\r'(-) \geqslant \alpha^+(\varphi_2)$ and~\cite[Corollary~3.6.3]{CluckersHalupczok}, we can write
    \begin{align*}
        &[T' \otimes \pi_{\Lambda_\n \times D'}^*\varphi_2 \cdot p'^*(\cf_{E'}\cf_{(\Lambda_\n \times B' \times h[n_1, 0, 0]) \setminus B_{N'}})](\lambda, (x_0, \xi_0), x', r', \xi, (x, y)) \\
        &\qquad = [\rho^*T' \cdot \sigma^*\varphi_2 \cdot p'^*(\cf_{E'}\cf_{(\Lambda_\n \times B' \times h[n_1, 0, 0]) \setminus B_{N'}})](\lambda, (x_0, \xi_0), x', r', \xi, (x, y)) \\
        &\qquad= T'(\lambda, (x_0, \xi_0), x', r', \xi, x) \cdot \LL^{n_2r'} \int_{h[n_2, 0, 0]} \varphi_2(z_2)\cf_{\Ball(0, r')^{n_2}}(y - z_2) \dd z_2 \\
        &\qquad\qquad\qquad \cdot [\cf_{E'}\cf_{(\Lambda_\n \times B' \times h[n_1, 0, 0]) \setminus B_{N'}}](\lambda, (x_0, \xi_0), x', r', \xi) \\
        &\qquad= \LL^{nr'} \int_{h[n, 0, 0]} \varphi_2(z_2)\cf_{\Ball(0, \alpha^-(\varphi_2))^{n_2}}(z_2)T'(\lambda, (x_0, \xi_0), x', r', \xi, x) \\
        &\qquad\qquad\qquad \cf_{\Ball(0, r')^{n_1}}(z_1)\cf_{\Ball(0, r')^{n_2}}(y - z_2) \\
        &\qquad\qquad\qquad \cdot [\cf_{E'}\cf_{(\Lambda_\n \times B' \times h[n_1, 0, 0]) \setminus B_{N'}}](\lambda, (x_0, \xi_0), x', r', \xi) \dd z_1 \dd z_2 \\
        &\qquad= [\LL^{nr'}\pi_{2!}((\sigma \circ \pi_1)^*\varphi_2 \cdot (\rho \circ \pi_2)^*T' \cdot \pi_1^*f \cdot g \\
        &\qquad\qquad\qquad \cdot h \cdot (p' \circ \pi_2)^*(\cf_{E'}\cf_{(\Lambda_\n \times B' \times h[n_1, 0, 0]) \setminus B_{N'}}))](\lambda, (x_0, \xi_0), x', r', \xi, (x, y)).
    \end{align*}
    Consequently, by~\cite[Theorem~1]{CluckersHalupczok}, since $\pi_2 = p' \times \id_{h[n, 0, 0]}$, the function~\eqref{eq:function} is
    \begin{multline}\label{eq:function-eq}
        T' \otimes \pi_{\Lambda_\n \times D'}^*\varphi_2 \cdot p'^*(\cf_{E'}\cf_{(\Lambda_\n \times B' \times h[n_1, 0, 0]) \setminus B_{N'}}) \\
        = \LL^{nr'}(p' \times \id_{h[n, 0, 0]})_!((\sigma \circ \pi_1)^*\varphi_2 \cdot (\rho \circ \pi_2)^*T' \cdot \pi_1^*f \cdot g \\
        \cdot h \cdot (p' \circ \pi_2)^*(\cf_{E'}\cf_{(\Lambda_\n \times B' \times h[n_1, 0, 0]) \setminus B_{N'}})).
    \end{multline}

    Then thanks to formula~\eqref{eq:function-eq}, the pushforward condition~\ref{it:pushforward} for the distribution~$u$ and the~$\Cons(\Lambda_\n \times D'[n, 0, 0])\expe$-linearity of the map $u_{\Lambda_\n \times D'[n, 0, 0]}$, we obtain
    \begin{align*}
        &\inner{u_{\varphi_2, \Lambda_\n \times D'}}{T'}\cf_{E'}\cf_{(\Lambda_\n \times B' \times h[n_1, 0, 0]) \setminus B_{N'}} \\
        &\qquad= \inner{u_{\Lambda_\n \times D'}}{T' \otimes \pi_{\Lambda_\n \times D'}^*\varphi_2 \cdot p'^*(\cf_{E'}\cf_{(\Lambda_\n \times B' \times h[n_1, 0, 0]) \setminus B_{N'}})} \\
        &\qquad= \inner{u_{\Lambda_\n \times D'}}{\LL^{nr'}(p' \times \id_{h[n, 0, 0]})_!((\sigma \circ \pi_1)^*\varphi_2 \cdot (\rho \circ \pi_2)^*T' \\
        &\qquad\qquad\qquad \cdot \pi_1^*f \cdot g \cdot h \cdot (p' \circ \pi_2)^*(\cf_{E'}\cf_{(\Lambda_\n \times B' \times h[n_1, 0, 0]) \setminus B_{N'}}))} \\
        &\qquad= \LL^{nr'}p'_!\inner{u_{\Lambda_\n \times D'[n, 0, 0]}}{(\sigma \circ \pi_1)^*\varphi_2 \cdot (\rho \circ \pi_2)^*T' \\
        &\qquad\qquad\qquad \cdot \pi_1^*f \cdot g \cdot h \cdot (p' \circ \pi_2)^*(\cf_{E'}\cf_{(\Lambda_\n \times B' \times h[n_1, 0, 0]) \setminus B_{N'}})} \\
        &\qquad= \LL^{nr'}p'_!(\sigma^*\varphi_2 \cdot f \cdot \inner{u_{\Lambda_\n \times D'[n, 0, 0]}}{(\rho \circ \pi_2)^*T' \\
        &\qquad\qquad\qquad \cdot g \cdot h \cdot (p' \circ \pi_2)^*(\cf_{E'}\cf_{(\Lambda_\n \times B' \times h[n_1, 0, 0]) \setminus B_{N'}})}).
    \end{align*}
    So to prove equation~\eqref{eq:msdd-2}, it is enough to show the equality
    \begin{equation}\label{eq:msdd-3}
        \inner{u_{\Lambda_\n \times D'[n, 0, 0]}}{(\rho \circ \pi_2)^*T' \cdot g \cdot h \cdot (p' \circ \pi_2)^*(\cf_{E'}\cf_{(\Lambda_\n \times B' \times h[n_1, 0, 0]) \setminus B_{N'}})} = 0.
    \end{equation}

    We consider the definable subassignments
    \begin{align*}
        D &\coloneq h[3n, 0, 0] \times h[0, 0, 1] \times h[n, 0, 0], \\
        E &\coloneq \{(\lambda, (x_0, y_0, \xi_0, \eta_0), (x', y'), r', (\xi, \eta)) \in \Lambda_\n \times B \times h[n, 0, 0] \\
        &\qquad\qquad \mid (\xi, \eta) \in \Ball((\xi_0, \eta_0), \check\r(x_0, y_0, \xi_0, \eta_0))\}, \\
        B_N &\coloneq \{(\lambda, (x_0, y_0, \xi_0, \eta_0), (x', y'), r', (\xi, \eta)) \in \Lambda_\n \times B \times h[n, 0, 0] \\
        &\qquad\qquad \mid \ord\lambda \geqslant N((x_0, y_0, \xi_0, \eta_0), (x', y'), r')\}
    \end{align*}
    and the function $T$ of $\Sch_{\Lambda_\n \times D}(\Lambda_\n \times D[n, 0, 0])\expe$ defined by the equality
    \[ T(\lambda, (x_0, y_0, \xi_0, \eta_0), (x', y'), r', (\xi, \eta), (x, y)) \coloneq \cf_{\Ball((x', y'), r')}(x, y)\Exp(\inner{(x, y)}{\lambda(\xi, \eta)}) \]
    for all points $(\lambda, (x_0, y_0, \xi_0, \eta_0), (x', y'), r', (\xi, \eta), (x, y))$ of $\Lambda_\n \times D[n, 0, 0]$. Since the quadruplet $(\A, \r, \check\r, N)$ is a $\Lambda_\n$-microlocally smooth definable data of the distribution~$u$, equation~\eqref{eq:msdd-comp} gives
    \begin{equation}\label{eq:msdd-u}
        \inner{u_{\Lambda_\n \times D}}{T}\cf_E\cf_{(\Lambda_\n \times B \times h[n, 0, 0]) \setminus B_N} = 0.
    \end{equation}
    Considering the canonical projection $p \colon \Lambda_\n \times D[n, 0, 0] \to \Lambda_\n \times D$ and the definable subassignment
    \begin{multline*}
        A \coloneq \\
        \left\{(\lambda, (x_0, y_0, \xi_0, \eta_0), (x', y'), r', (\xi, \eta)) \in \Lambda_\n \times D \mathrel{}\middle|\mathrel{} \begin{aligned}
            (x_0, \xi_0) &\in \A', \\
            \Ball(x', r') &\subset \Ball(x_0, \r'(x_0, \xi_0)), \\
            \xi &\in \Ball(\xi_0, \check\r'(x_0, \xi_0)), \\
            \ord\lambda &< N'((x_0, \xi_0), x', r'), \\
            y_0 &\in \Ball(0, \alpha^-(\varphi_2))^{n_2}
        \end{aligned}\right\}\kern-\nulldelimiterspace,
    \end{multline*}
    multiplying equation~\eqref{eq:msdd-u} by the function $\cf_A$ and using the $\Cons(\Lambda_\n \times D)\expe$-linearity of the map $u_{\Lambda_\n \times D}$, we obtain
    \begin{equation}\label{eq:msdd-u2}
        \inner{u_{\Lambda_\n \times D}}{T \cdot p^*(\cf_E\cf_{(\Lambda_\n \times B \times h[n, 0, 0]) \setminus B_N}\cf_A)} = 0.
    \end{equation}
    We consider the definable morphism
    \[ \fonc{\Phi}{\Lambda_\n \times D'[n, 0, 0]}{\Lambda_\n \times D,}{(\lambda, (x_0, \xi_0), x', r', \xi, (z_1, z_2))}{(\lambda, (x_0, z_2, \xi_0, 0), (x', z_2), r', (\xi, 0)).} \]
    With the pullback condition~\ref{it:pullback} for the distribution $u$, equality~\eqref{eq:msdd-u2} gives
    \[ \inner{u_{\Lambda_\n \times D'[n, 0, 0]}}{(\Phi \times \id_{h[n, 0, 0]})^*(T \cdot p^*(\cf_E\cf_{(\Lambda_\n \times B \times h[n, 0, 0]) \setminus B_N}\cf_A))} = 0. \]
    So to prove equation~\eqref{eq:msdd-3}, it is enough to show the equality
    \begin{multline*}
        (\rho \circ \pi_2)^*T' \cdot g \cdot h \cdot (p' \circ \pi_2)^*(\cf_{E'}\cf_{(\Lambda_\n \times B' \times h[n_1, 0, 0]) \setminus B_{N'}}) \\
        = (\Phi \times \id_{h[n, 0, 0]})^*(T \cdot p^*(\cf_E\cf_{(\Lambda_\n \times B \times h[n, 0, 0]) \setminus B_N}\cf_A)).
    \end{multline*}
    For that, it is enough to show these two equalities
    \begin{equation}\label{eq:msdd-31}
        (\rho \circ \pi_2)^*T' \cdot g = (\Phi \times \id_{h[n, 0, 0]})^*T
    \end{equation}
    and
    \begin{multline}
        h \cdot (p' \circ \pi_2)^*(\cf_{E'}\cf_{(\Lambda_\n \times B' \times h[n_1, 0, 0]) \setminus B_{N'}}) \\
        = (p \circ (\Phi \times \id_{h[n, 0, 0]}))^*(\cf_E\cf_{(\Lambda_\n \times B \times h[n, 0, 0]) \setminus B_N}\cf_A). \label{eq:msdd-32}
    \end{multline}

    First, let us prove equality~\eqref{eq:msdd-31}. Let $(\lambda, (x_0, \xi_0), x', r', \xi, (z_1, z_2), (x, y))$ be a point of $\Lambda_\n \times D'[n, 0, 0][n, 0, 0]$. By definition of the morphism $g$, we get
    \begin{align*}
        &[(\Phi \times \id_{h[n, 0, 0]})^*T](\lambda, (x_0, \xi_0), x', r', \xi, (z_1, z_2), (x, y)) \\
        &\qquad= T(\lambda, (x_0, z_2, \xi_0, 0), (x', z_2), r', (\xi, 0), (x, y)) \\
        &\qquad= \cf_{\Ball((x', z_2), r')}(x, y)\Exp(\inner{(x, y)}{\lambda(\xi, 0)}) \\
        &\qquad= \cf_{\Ball(x', r')}(x)\Exp(\inner{x}{\lambda\xi}) \cdot \cf_{\Ball(z_2, r')}(y) \\
        &\qquad= T'(\lambda, (x_0, \xi_0), x', r', \xi, x) \cdot g(\lambda, (x_0, \xi_0), x', r', \xi, (z_1, z_2), (x, y)) \\
        &\qquad= [(\rho \circ \pi_2)^*T' \cdot g](\lambda, (x_0, \xi_0), x', r', \xi, (z_1, z_2), (x, y))
    \end{align*}
    which, by~\cite[Theorem~1]{CluckersHalupczok}, leads to equality~\eqref{eq:msdd-31}.

    Next, let us prove equality~\eqref{eq:msdd-32}. Let $(\lambda, (x_0, \xi_0), x', r', \xi, (z_1, z_2), (x, y))$ be a point of $\Lambda_\n \times D'[n, 0, 0][n, 0, 0]$. On the one hand, by definition of the morphism $h$, we get
    \begin{multline*}
        [h \cdot (p' \circ \pi_2)^*(\cf_{E'}\cf_{(\Lambda_\n \times B' \times h[n_1, 0, 0]) \setminus B_{N'}})](\lambda, (x_0, \xi_0), x', r', \xi, (z_1, z_2), (x, y)) \\
        = \cf_{\Ball(0, \alpha^-(\varphi_2))^{n_2}}(z_2)\cf_{E' \cap [(\Lambda_\n \times B' \times h[n_1, 0, 0]) \setminus B_{N'}]}(\lambda, (x_0, \xi_0), x', r', \xi)
    \end{multline*}
    and, on the other hand, we get
    \begin{multline*}
        [(p \circ (\Phi \times \id_{h[n, 0, 0]}))^*(\cf_E\cf_{(\Lambda_\n \times B \times h[n, 0, 0]) \setminus B_N}\cf_A)](\lambda, (x_0, \xi_0), x', r', \xi, (z_1, z_2), (x, y)) \\
        = \cf_{E \cap [(\Lambda_\n \times B \times h[n, 0, 0]) \setminus B_N] \cap A}(\lambda, (x_0, z_2, \xi_0, 0), (x', z_2), r', (\xi, 0)).
    \end{multline*}
    We can write
    \begin{multline}\label{eq:equivalence-1}
        \Sys{
            (\lambda, (x_0, \xi_0), x', r', \xi) &\in E' \cap [(\Lambda_\n \times B' \times h[n_1, 0, 0]) \setminus B_{N'}], \\
            z_2 &\in \Ball(0, \alpha^-(\varphi_2))^{n_2}
        } \\
        \iff \Sys{
            (x_0, \xi_0) &\in \A', \\
            \Ball(x', r') &\subset \Ball(x_0, \r'(x_0, \xi_0)), \\
            \xi &\in \Ball(\xi_0, \check\r'(x_0, \xi_0)), \\
            \ord\lambda &< N'((x_0, \xi_0), x', r'), \\
            z_2 &\in \Ball(0, \alpha^-(\varphi_2))^{n_2}
        }
    \end{multline}
    and
    \begin{multline}\label{eq:equivalence-2}
        (\lambda, (x_0, z_2, \xi_0, 0), (x', z_2), r', (\xi, 0)) \in E \cap [(\Lambda_\n \times B \times h[n, 0, 0]) \setminus B_N] \cap A \\
        \iff \Sys{
            (x_0, z_2, \xi_0, 0) &\in \A, \\
            \Ball((x', z_2), r') &\subset \Ball((x_0, z_2), \r(x_0, z_2, \xi_0, 0)), \\
            (\xi, 0) &\in \Ball((\xi_0, 0), \check\r(x_0, z_2, \xi_0, 0)), \\
            \ord\lambda &< N((x_0, z_2, \xi_0, 0), (x', z_2), r'), \\
            (x_0, \xi_0) &\in \A', \\
            \Ball(x', r') &\subset \Ball(x_0, \r'(x_0, \xi_0)), \\
            \xi &\in \Ball(\xi_0, \check\r'(x_0, \xi_0)), \\
            \ord\lambda &< N'((x_0, \xi_0), x', r'), \\
            z_2 &\in \Ball(0, \alpha^-(\varphi_2))^{n_2}
        } \\
        \iff \Sys{
            (x_0, z_2, \xi_0, 0) &\in \A, \\
            \Ball(x', r') &\subset \Ball(x_0, \r(x_0, z_2, \xi_0, 0)), \\
            r' &\geqslant \r(x_0, z_2, \xi_0, 0), \\
            \xi &\in \Ball(\xi_0, \check\r(x_0, z_2, \xi_0, 0)), \\
            \ord\lambda &< N((x_0, z_2, \xi_0, 0), (x', z_2), r'), \\
            (x_0, \xi_0) &\in \A', \\
            \Ball(x', r') &\subset \Ball(x_0, \r'(x_0, \xi_0)), \\
            \xi &\in \Ball(\xi_0, \check\r'(x_0, \xi_0)), \\
            \ord\lambda &< N'((x_0, \xi_0), x', r'), \\
            z_2 &\in \Ball(0, \alpha^-(\varphi_2))^{n_2}.
        }
    \end{multline}
    We want to prove the equivalence $\eqref{eq:equivalence-1} \iff \eqref{eq:equivalence-2}$. The implication~$\eqref{eq:equivalence-2} \implies \eqref{eq:equivalence-1}$ is immediate. Reciprocally, we assume that the condition~\eqref{eq:equivalence-1} is fulfilled. We need to verify the first five conditions of equation~\eqref{eq:equivalence-2}.
    \begin{itemize}
        \item Since $z_2 \in \Ball(0, \alpha^-(\varphi_2))^{n_2}$ and $(x_0, \xi_0) \in \A'$, we get $(x_0, z_2, \xi_0, 0) \in \A$.

        \item Since $\r'(x_0, \xi_0) \geqslant \r(x_0, z_2, \xi_0, 0)$ and $\Ball(x', r') \subset \Ball(x_0, \r'(x_0, \xi_0))$, we get
        \[ \Ball(x', r') \subset \Ball(x_0, \r(x_0, z_2, \xi_0, 0)). \]

        \item Since $\Ball(x', r') \subset \Ball(x_0, \r'(x_0, \xi_0))$, we get $r' \geqslant \r'(x_0, \xi_0) \geqslant \r(x_0, z_2, \xi_0, 0)$.

        \item Since $\check\r'(x_0, \xi_0) \geqslant \check\r(x_0, z_2, \xi_0, 0)$ and $\xi \in \Ball(\xi_0, \check \r'(x_0, \xi_0))$, we get
        \[ \xi \in \Ball(\xi_0, \check\r(x_0, z_2, \xi_0, 0)). \]

        \item Since
        \[ N'((x_0, \xi_0), x', r') \leqslant N((x_0, z_2, \xi_0, 0), (x', z_2), r') \]
        and $\ord\lambda < N'((x_0, \xi_0), x', r')$, we get
        \[ \ord\lambda < N((x_0, z_2, \xi_0, 0), (x', z_2), r'). \]
    \end{itemize}
    This implied equation~\eqref{eq:equivalence-1} and so the desired equivalence. Consequently, equality~\eqref{eq:msdd-32} is shown.

    In conclusion, we have proven equalities~\eqref{eq:msdd-31} and~\eqref{eq:msdd-32} which imply equality~\eqref{eq:msdd-3} and so equality~\eqref{eq:msdd-2}. This concludes the proof.
\end{proof}

\subsubsection{Wave front sets and kernels}

We give now the analog of Theorem~\ref{thm:WF-mot} with wave front sets. In the following, we will consider points instead of definable subassignments. Then we will always work over the residue fields of the considered points. For the next definition, we take a general positive integer $n$.

\begin{defi}[Raibaut~{\cite[Definition~6.9]{Raibaut}}]
    A point $(x_0, \xi_0, K)$ of $h[n, 0, 0] \times \B_\n^n$ is a \emph{$\Lambda_\n$-microlocally smooth point} of a distribution~$u$ of $\Sch'(h[n, 0, 0])\expe$ if there exist some integers $\r$ and $\check\r$ greater than~$\n$ and a definable morphism $N \colon B \to h[0, 0, 1]$ with
    \[ B \coloneq \{(x', r') \in h[n, 0, 0] \times h[0, 0, 1] \mid \Ball(x', r') \subset \Ball(x_0, \r)\} \]
    such that
    \[ \inner{u_{\Lambda_\n \times D}}{T}\cf_E\cf_{B_N} = \inner{u_{\Lambda_\n \times D}}{T}\cf_E \]
    where we consider the definable subassignments of $\Def_K$
    \begin{align*}
        D &\coloneq h[n, 0, 0] \times h[0, 0, 1] \times h[n, 0, 0], \\
        E &\coloneq \{(\lambda, x', r', \xi) \in \Lambda_\n \times B \times h[n, 0, 0] \mid \xi \in \Ball(\xi_0, \check\r)\}, \\
        B_N &\coloneq \{(\lambda, x', r', \xi) \in \Lambda_\n \times B \times h[n, 0, 0] \mid \ord\lambda \geqslant N(x', r')\}
    \end{align*}
    and the function $T$ of $\Sch_{\Lambda_\n \times D}(\Lambda_\n \times D[n, 0, 0])\expe$ defined by the equality
    \[ T(\lambda, x', r', \xi, x) \coloneq \cf_{\Ball(x', r')}(x)\Exp(\inner{x}{\lambda\xi}) \]
    for all points $(\lambda, x', r', \xi, x)$ of $\Lambda_\n \times D[n, 0, 0]$.

    A point $(x_0, \xi_0, K)$ of $h[n, 0, 0] \times (h[n, 0, 0] \setminus \{0\})$ is $\Lambda_\n$-microlocally smooth point of the distribution $u$ if there exists an element $\lambda$ of $\Lambda_\n(K)$ such that $\lambda\xi_0 \in \B_\n^n$ and the point $(x_0, \lambda\xi_0, K)$ is $\Lambda_\n$-microlocally smooth point of the distribution $u$.

    The complement of the set of points of $h[n, 0, 0] \times (h[n, 0, 0] \setminus \{0\})$ at which the distribution $u$ is $\Lambda$-microlocally smooth is the \emph{$\Lambda_\n$-wave front set} of the distribution~$u$, denoted by $\WF_{\Lambda_\n}(u)$.
\end{defi}

We get immediately the following corollary of Theorem~\ref{thm:WF-mot}. We take again the same notation of~§\ref{ssec:Schwartz}.

\begin{coro}\label{coro:WF-mot}
    Let $u$ be a distribution of $\Sch'(h[n, 0, 0])\expe$. Let $\varphi_2$ be a function of~$\Sch(h[n_2, 0, 0])\expe$ and let $u_{\varphi_2}$ be the induced distribution of Proposition~\ref{prop:induced}. We consider
    \begin{itemize}
        \item the set $S$ of definable subassignments $\A \subset h[n, 0, 0] \times \B_\n^n$ of $\Lambda_\n$-microlocally smooth definable data $(\A, \r, \check\r, N)$ of the distribution $u$;

        \item for such a definable subassignment $\A$ of $S$ and for a choice of a definable morphism $\alpha^-(\varphi_2)$, the definable subassignment $\A'$ of~$h[n_1, 0, 0] \times \B_\n^{n_1}$ constructed in Theorem~\ref{thm:WF-mot}, that is
        \begin{multline*}
            \A'(\alpha^-(\varphi_2)) \coloneq \{(x_0, \xi_0) \in h[n_1, 0, 0] \times \B_\n^{n_1} \\
            \mid \forall y \in \Ball(0, \alpha^-(\varphi_2))^{n_2}, (x_0, y, \xi_0, 0) \in \A\}.
        \end{multline*}
    \end{itemize}
    Then
    \[ \WF_{\Lambda_\n}(u_{\varphi_2}) \cap \points{h[n_1, 0, 0] \times \B_\n^{n_1}} \subset \bigcap_{\substack{\A \in S \\ \alpha^-(\varphi_2)}} \points{(h[n_1, 0, 0] \times \B_\n^{n_1}) \setminus \A'(\alpha^-(\varphi_2))}. \]
\end{coro}

\begin{proof}
    It follows from Theorem~\ref{thm:WF-mot} and~\cite[Remark~6.10]{Raibaut}.
\end{proof}

\begin{rema}
    The proof of Theorem~\ref{thm:extension} uses product of distributions which, in the motivic setting, is only defined locally (namely up to a microlocally smooth definable data) with existential condition on the tensor product (see~\cite[Definition~6.31]{Raibaut}). Thus, a local motivic version of Theorem~\ref{thm:extension} should be obtained, which will be done in a future work.
\end{rema}

\printbibliography

\end{document}